%% file: nonsym_arXiv.tex
\documentclass[12pt]{article}

\usepackage{color} % Need the color package
\usepackage{epsfig}

\usepackage{amsmath,amssymb,amsfonts,amsthm}
\usepackage{moreverb,rotating,graphics}
 \usepackage[T1]{fontenc}
 \usepackage[normalem]{ulem}
 \usepackage{verbatim}
\usepackage[utf8]{inputenc}

\usepackage{latexsym}
\usepackage{pdflscape}
\usepackage{examplep}
\usepackage{longtable}
\usepackage{graphicx}
\usepackage{amsmath}
\usepackage{booktabs}
\usepackage{memhfixc}

\usepackage[francais,english]{babel} 

\newtheorem{theorem}{Theorem}[section]

\newtheorem{remark}{Remark}[section]
\newtheorem{lemma}{Lemma}[section]

\newtheorem{example}{Example}[section]

\newtheorem{prop}{Proposition}[section]

\def\cC{{\mathcal C}}

\def\N{{\mathbb N}}
\def\Z{{\mathbb Z}}
\def\R{{\mathbb R}}
\def\C{{\mathbb C}}

\def\cC{{\cal C}}
\def\cL{{\cal L}}

\def\cK{{\cal K}}
\def\cT{{\cal T}}

\def\cK{{\cal K}}

\def\cI{{\cal I}}

\def\cX{{\cal X}}
\def\cE{{\cal E}}

\def\cJ{{\cal J}}

\def\it{\widetilde{\iota}}
\def\Vt{\widetilde{V}}
\def\cAt{\widetilde{\mathcal{A}}}
\def\cBt{\widetilde{\mathcal{B}}}
\def\cUt{\widetilde{\mathcal{U}}}
\def\cLt{\widetilde{\mathcal{L}}}

\def\mt{\widetilde{m}}

\def\rh{\widehat{r}}
\def\sh{\widehat{s}}

\def\vt{\widetilde{v}}

\def\wt{\widetilde{w}}
\def\mt{\widetilde{m}}

\def\rt{\widetilde{r}}
\def\st{\widetilde{s}}
\def\Vt{\widetilde{V}}
\def\cD{\mathcal{D}}

\def\phih{\widehat{\phi}}
\def\cA{\mathcal{A}}
\def\cB{\mathcal{B}}

\newcommand \dps{\displaystyle }

\title{Greedy algorithms for high-dimensional non-symmetric linear problems}
\author{Eric Canc\`es, Virginie Ehrlacher, Tony Leli\`evre}
\begin{document}

%\selectlanguage{francais}
\selectlanguage{english}

\maketitle
%\begin{abstract}
%\end{abstract}
%\newpage
%\tableofcontents

\begin{abstract} 
In this article, we present a family of numerical approaches to solve high-dimensional linear non-symmetric problems. The principle of these 
methods is to approximate a function which depends on a large number of variates by a sum of tensor product functions, each term of which is iteratively computed via 
a greedy algorithm~\cite{Temlyakov}. There exists a good theoretical framework for these methods in the case of (linear and nonlinear) symmetric elliptic problems. However, the 
convergence results are not valid any more as soon as the problems considered are not symmetric. We present here a review of the main algorithms proposed in the literature to 
circumvent this difficulty, together with some new approaches. The theoretical convergence results and the practical implementation of these algorithms are discussed. 
Their behaviors are illustrated 
through some numerical examples.  
\end{abstract}

\section*{Introduction}

High-dimensional problems arise in a wide range of fields such as quantum chemistry, molecular dynamics, uncertainty quantification, polymeric fluids, finance...
 In all these contexts, 
one wishes to approximate a function $u$ depending on $d$ variates $x_1$, ..., $x_d$ where $d\in\N^*$ is typically very large. Classically, the function $u$ is defined as the 
solution of a Partial Differential Equation (PDE) and cannots be obtained by standard approximation techniques such as Galerkin methods for instance. Indeed, let us 
consider a discretization basis with $N$ degrees of freedom for each variate ($N\in \N^*$), so that the discretization space is given by
$$
V_N := \mbox{\rm Span}\left\{ \psi_{i_1}^{(1)}(x_1) \cdots \psi_{i_d}^{(d)}(x_d), \; 1\leq i_1, \cdots, i_d \leq N \right\},
$$
where for all $1\leq j \leq d$, $\left( \psi_i^{(j)} \right)_{1\leq i \leq N}$ is a family of $N$ functions which only depend on the variate $x_j$. A Galerkin method 
consists in representing 
the solution $u$ of the initial PDE as
$$
u(x_1, \cdots, x_d) \approx \sum_{1\leq i_1, \cdots, i_d \leq N} \lambda_{i_1, \cdots,i_d} \psi_{i_1}^{(1)}(x_1) \cdots \psi_{i_d}^{(d)}(x_d), 
$$
and computing the set of $N^d$ real numbers $\left( \lambda_{i_1,\cdots, i_d} \right)_{1\leq i_1, \cdots, i_d \leq N}$. Thus, the size of the 
finite-dimensional problem to solve grows exponentially with the number of variates involved in the problem. Such methods cannot be implemented when $d$ is too large:
 this is the so-called \itshape curse of dimensionality\normalfont~\cite{Bellman}.

Several approaches have recently been proposed in order to circumvent this significant difficulty. Let us mention among others sparse grids \cite{SchwabPeter}, 
tensor formats \cite{Hackbusch}, 
reduced bases \cite{Maday} and adaptive polynomial approximations \cite{Cohen2}. 

\medskip

In this paper, we will focus on a particular kind of methods, originally introduced by Ladev\`eze {\em et~al.} to do time-space variable separation \cite{Ladeveze}, 
Chinesta {\em et~al.} to solve high-dimensional Fokker-Planck equations in the context of kinetic models for polymers 
\cite{Chinesta} and Nouy in the context of uncertainty quantification \cite{Nouy}, under the name of \itshape Progressive Generalized Decomposition \normalfont (PGD) methods.

Let us assume that each variate $x_j$ belongs to a subset $\cX_j$ of $\R^{m_j}$, where $m_j\in\N^*$ for all $1\leq j\leq d$. For each $d$-uplet $(r^{(1)}, \cdots, r^{(d)})$ of functions 
such that $r^{(j)}$ only depends on $x_j$ for all $1\leq j\leq d$, we call a \itshape tensor product function \normalfont and denote by $r^{(1)}\otimes \cdots \otimes r^{(d)}$ 
the function which depends on all the variates $x_1, \cdots, x_d$ and is defined by
$$
r^{(1)}\otimes \cdots \otimes r^{(d)}: \left\{ \begin{array}{ccc}
                                                \cX_1 \times \cdots \times \cX_d & \to & \R\\
                                                (x_1, \cdots, x_d) & \mapsto & r^{(1)}(x_1) \cdots r^{(d)}(x_d).\\
                                               \end{array}
\right .
$$
 
The approach of Ladev\`eze, Chinesta, Nouy and coauthors consists in approximating the function $u$ by a separate variable decomposition, i.e.
\begin{equation}\label{eq:sumexp}
u(x_1, \cdots, x_d) \approx \sum_{k=1}^{n} r_k^{(1)}(x_1) \cdots r_k^{(d)}(x_d) = \sum_{k=1}^{n} r_k^{(1)}\otimes \cdots \otimes r_k^{(d)}(x_1, \cdots, x_d),
\end{equation}
for some $n\in\N^*$. In the above sum, each term is a tensor product function. Each 
$d$-uplet of functions $\left(r_k^{(1)}, \cdots, r_k^{(d)}\right)$ is iteratively computed in a \itshape greedy \normalfont \cite{Temlyakov} way: once the first 
$k$ terms in the sum (\ref{eq:sumexp}) have been computed, they are fixed, and the $(k+1)^{th}$ term is obtained as the \itshape next best tensor product function 
\normalfont to approximate the 
solution. This will be made precise below.  

Thus, the algorithm consists 
in solving several low-dimensional problems whose dimensions scale linearly with the number of variates and may be implementable when classical methods are not. 
In this case, if we use a discretization basis with $N$ degrees of freedom 
per variate as above, the size of the discretized problems involved in the computation of a $d$-uplet $\left(r_k^{(1)}, \cdots , r_k^{(d)}\right)$ scales like $Nd$ and the 
total size of the discretization problems is $nNd$. 

This numerical strategy has been extensively studied for the resolution of (linear or nonlinear) elliptic problems \cite{Temlyakov,LBLM,Figueroa,CELgreedy,NouyFalco}. More precisely, let $u$ be defined 
as the unique solution of a minimization problem of the form
\begin{equation}\label{eq:min}
u = \mathop{\mbox{\rm argmin}}_{v\in V} \mathcal{E}(v),
\end{equation}
where $V$ is a reflexive Banach space of functions depending on the $d$ variates $x_1$, ..., $x_d$, and $\mathcal{E}: V \to \R$ is a coercive real-valued energy functional. Besides, 
for all $1\leq j \leq d$, let $V_{x_j}$ be a reflexive Banach space of functions which only depend on the variate $x_j$.  
The standard greedy algorithm reads: 
\begin{enumerate}
 \item set $u_0 = 0$ and $n=1$; 
\item find $\left(r_n^{(1)}, \cdots, r_n^{(d)}\right) \in V_{x_1}\times \cdots \times V_{x_d}$ such that
$$
\left(r_n^{(1)}, \cdots, r_n^{(d)}\right) \in  \mathop{\mbox{argmin}}_{\left( r^{(1)}, \cdots, r^{(d)}\right) \in V_{x_1} \times \cdots \times V_{x_d}} \mathcal{E} \left( u_{n-1} + r^{(1)} \otimes \cdots \otimes r^{(d)} \right),
$$
\item set $u_n = u_{n-1} + r_n^{(1)}\otimes \cdots \otimes r_n^{(d)}$ and $n=n+1$.  
\end{enumerate}
Under some natural assumptions on the spaces $V$, $V_{x_1}$, ..., $V_{x_d}$ and the energy functional $\mathcal{E}$, all the iterations of the greedy algorithm are well-defined and the sequence 
$(u_n)_{n\in\N^*}$ strongly converges in $V$ towards the solution $u$ of the original minimization problem (\ref{eq:min}).

This result holds in particular when 
$u$ is defined as the unique solution of 
$$
\left\{
\begin{array}{l}
 \mbox{\rm find } u \in V \mbox{\rm such that}\\
\forall v\in V, \; a(u,v) = l(v),\\
\end{array}
\right .
$$
where $V$ is a Hilbert space, $a$ a \itshape symmetric \normalfont continuous coercive bilinear form on $V\times V$ and $l$ a continuous linear form on $V$. In this case, $u$ is equivalently solution of a 
minimization problem of the form (\ref{eq:min}) with $\mathcal{E}(v) = \frac{1}{2}a(v,v) - l(v)$ for all $v\in V$. 

\medskip

However, when the function $u$ cannot be defined as the solution of a minimization problem of the form (\ref{eq:min}), designing efficient iterative algorithms is not an obvious task. 
This situation occurs typically when $u$ is defined as the solution of a \itshape non-symmetric \normalfont linear problem
$$
\left\{
\begin{array}{l}
 \mbox{\rm find } u \in V \mbox{\rm such that}\\
\forall v\in V, \; a(u,v) = l(v),\\
\end{array}
 \right .
$$
where $a$ is a non-symmetric continuous bilinear form on $V\times V$ and $l$ is a continuous linear form on $V$.

The aim of this article is to give an overview of the state of the art of the numerical methods based on the greedy iterative approach used in this non-symmetric linear context and of the
 remaining open questions concerning this issue. In Section~\ref{sec:sym}, we present the standard greedy algorithm for the resolution of symmetric coercive 
high-dimensional problems and the theoretical convergence results proved in this 
setting. Section~\ref{sec:nonsym} explains why a naive transposition of this algorithm for non-symmetric problems is doomed to failure and motivates the need 
for more subtle approaches. 
Section~\ref{sec:resmin} describes the certified algorithms existing in the literature for non-symmetric problems. All of them consist in \itshape symmetrizing \normalfont
 the original non-symmetric problem 
by minimizing the residual of the equation in a well-chosen norm. However, depending on the choice of the norm, either the conditioning of the discretized problems may behave badly 
or several intermediate problems may have to be solved online, which leads to a significant increase of simulation times and memory needs compared to the original algorithm in a symmetric linear coercive case.
 So far, there are no methods avoiding these two problems and 
for which there are theoretical convergence results in the general case. In Section~\ref{sec:dual}, we present some existing algorithms designed by Nouy \cite{NouyMinMax} and Lozinski \cite{Lozinski} to circumvent 
these difficulties and the partial theoretical results which are known for these algorithms. Section~\ref{sec:us} is concerned with another algorithm we propose, for which some 
partial convergence results are proved. In Section~\ref{sec:num}, the behaviors of the different algorithms presented here are illustrated on simple toy numerical examples. 
Lastly, we present in the Appendix some possible tracks to design other methods, but for which further work is needed.

\section{The symmetric coercive case}\label{sec:sym}

\subsection{Notation}\label{sec:notation}

Let us first introduce some notation. Let $d$ be a positive integer, $m_1$, ..., $m_d$ positive integers and $\cX_1$, ...,$\cX_d$ open subsets of $\R^{m_1}$, ..., $\R^{m_d}$ 
respectively. 

Let $\mu_{x_1}$, ..., $\mu_{x_d}$ denote measures on $\cX_1$, ..., $\cX_d$ respectively. Let $L^2(\cX_1; \mu_{x_1})$, ..., 
$L^2(\cX_d; \mu_{x_d})$ be associated $L^2$ spaces, i.e. vectorial spaces which are complete when endowed with the scalar products
\begin{align*}
 \forall f,g\in L^2(\cX_j; \mu_{x_j}), \quad & \langle f,g \rangle_{\cX_j} := \int_{\cX_j} f(x_j)g(x_j) \, \mu_{x_j}(dx_j),\; \forall 1\leq j \leq d,\\
\end{align*}
and their associated norms $\|\cdot\|_{\cX_1}$, ..., $\|\cdot\|_{\cX_d}$. For instance, in the case when $\cX_1 = (0,1)$ and $\mu_{x_1}$ is the standard Lebesgue measure on $\cX_1$, 
the spaces $L^2(0,1)$, $L^2_{\rm per}(0,1)$ and $L^2_0(0,1):=\left\{ f \in L^2(0,1), \; \int_0^1f = 0\right\}$ are examples of such $L^2$ spaces.

In the rest of this article, for the sake of simplicity, we will omit the reference to the measures $\mu_{x_1}$, ..., $\mu_{x_d}$ and denote by
 $L^2(\cX_1) = L^2(\cX_1; \mu_{x_1})$, ..., $L^2(\cX_d) = L^2(\cX_d; \mu_{x_d})$.

\medskip

We introduce the space 
$L^2( \cX_1\times \cdots \times \cX_d) := L^2(\cX_1) \otimes \cdots \otimes L^2(\cX_d)$. 
This space is a Hilbert space when endowed with the natural scalar product
$$
\forall f,g\in L^2(\cX_1\times \cdots \times \cX_d),\quad \langle f,g \rangle := \int_{\cX_1 \times \cdots \times \cX_d}f(x_1, \cdots , x_d)g(x_1, \cdots , x_d)\,\mu_{x_1}(dx_1) \, \cdots \mu_{x_d}(dx_d),
$$
and the associated norm is denoted by $\|\cdot \|_{\cX_1 \times \cdots \times \cX_d}$. 

\medskip

Let $V\subset L^2( \cX_1\times \cdots \times \cX_d)$, $V_{x_1}\subset L^2(\cX_1)$, ..., $V_{x_d}\subset L^2(\cX_d)$ 
be Hilbert spaces endowed respectively 
with scalar products denoted by $\langle \cdot, \cdot\rangle_V$, $\langle \cdot, \cdot \rangle_{V_{x_1}}$, 
..., $\langle \cdot, \cdot \rangle_{V_{x_d}}$ and associated norms $\|\cdot\|_V$, $\|\cdot\|_{V_{x_1}}$, ..., $\|\cdot\|_{V_{x_d}}$. 

We define $V'$, $V_{x_1}'$, ..., $V_{x_d}'$ as the dual spaces of $V$, $V_{x_1}$, ..., $V_{x_d}$ 
with respect to the $L^2$ scalar products $\langle \cdot, \cdot \rangle$, $\langle \cdot, \cdot \rangle_{\cX_1}$, ..., $\langle \cdot, \cdot \rangle_{\cX_d}$.
 These dual spaces are endowed with their natural norms $\|\cdot\|_{V'}$ etc.  

Lastly, the Riesz operator $R_V: V \to V'$ is defined by
$$
\forall v,w\in V, \;  \langle v, w\rangle_V  = \langle R_V v,w\rangle_{V',V}.
$$
It holds in particular that $\|v\|_V = \|R_V v\|_{V'}$. Similar operators $R_{V_{x_1}}$, ..., $R_{V_{x_d}}$ 
are introduced for the spaces $V_{x_1}$, ..., $V_{x_d}$. 

\medskip

For any $d$-uplet $\left( r^{(1)}, \cdots, r^{(d)}\right) \in V_{x_1} \times \cdots \times V_{x_d}$, we define the \itshape tensor product function \normalfont $r^{(1)} \otimes \cdots \otimes r^{(d)}$ as follows
$$
r^{(1)} \otimes \cdots \otimes r^{(d)} : \left\{ 
\begin{array}{ccc}
 \cX_1 \times \cdots \times \cX_d & \to & \R \\
(x_1, \cdots,x_d) & \mapsto & r^{(1)}(x_1) \cdots r^{(d)}(x_d).\\
\end{array}
\right .
$$

\medskip

In the particular case when $d=2$, we shall denote respectively $x_1$, $\cX_1$, $m_1$, $V_{x_1}$ by $x$, $\cX$, $m_x$, $V_x$ and $x_2$, $\cX_2$, $m_2$, $V_{x_2}$ by $t$, $\cT$, $m_t$, $V_t$.

\medskip

Besides, for any Banach spaces $H_1$, $H_2$, the space of bounded linear operators from $H_1$ to $H_2$ will be denoted by $\mathfrak{L}(H_1, H_2)$. 
 
 \subsection{Theoretical results}

We recall here the theoretical framework of the standard greedy algorithm in the coercive symmetric case.

\medskip

Let us consider the problem
\begin{equation}\label{eq:sym}
\left\{
\begin{array}{l}
 \mbox{find }u\in V\mbox{ such that}\\
\forall v\in V, \; a(u,v) = l(v),\\
\end{array}
\right.
\end{equation}
where
\begin{itemize}
 \item $a(\cdot, \cdot)$ is a \itshape symmetric, coercive \normalfont continuous bilinear form on $V\times V$;
\item $l$ is a continuous linear form on $V$.
\end{itemize}

Then, problem (\ref{eq:sym}) is equivalent to the minimization problem
\begin{equation}\label{eq:minpb}
u  = \mathop{\mbox{argmin}}_{v\in V} \cE(v),
\end{equation}
where 
\begin{equation}\label{eq:symener}
 \forall v\in V, \; \cE(v) := \frac{1}{2}a(v,v) - l(v).
\end{equation}

\medskip

The greedy algorithm reads: 
\begin{enumerate}
 \item let $u_0 = 0$ and $n=1$; 
\item define $\left(r_n^{(1)}, \cdots, r_n^{(d)} \right) \in V_{x_1} \times \cdots \times V_{x_d}$ such that 
\begin{equation}\label{eq:algomin}
\left(r_n^{(1)}, \cdots, r_n^{(d)} \right) \in \mathop{\mbox{\rm argmin}}_{\left( r^{(1)}, \cdots, r^{(d)} \right) \in V_{x_1} \times \cdots \times V_{x_d}} \cE\left( u_{n-1} + r^{(1)} \otimes \cdots \otimes r^{(d)} \right);
\end{equation}
\item define $u_n = u_{n-1} + r_n^{(1)} \otimes \cdots \otimes r_n^{(d)}$ and set $n=n+1$. 
\end{enumerate}

Let us denote by 
\begin{equation}\label{eq:defSigma}
\Sigma := \left\{ r^{(1)} \otimes \cdots \otimes r^{(d)}, \; r^{(1)}\in V_{x_1}, \cdots, r^{(d)} \in V_{x_d} \right\}
\end{equation}
and make the following assumptions:
\begin{itemize}
 \item[(A1)] $\overline{\mbox{\rm Span}(\Sigma)}^V = V$;
\item[(A2)] $\Sigma$ is weakly closed in $V$.
\end{itemize}
These assumptions are usually satisfied in the case of classical Sobolev spaces~\cite{CELgreedy}.

\begin{theorem}\label{th:sym}
Assume that (A1) and (A2) are satisfied. Then, for all $n\in\N^*$, there exists at least one solution 
$\left(r^{(1)}_n,\cdots, r_n^{(d)}\right)\in V_{x_1}\times \cdots \times V_{x_d}$ 
(not necessarily unique) to (\ref{eq:algomin}) and any solution satisfies $r^{(1)}_n\otimes \cdots \otimes r^{(d)}_n \neq 0$ if and only if $u_{n-1} \neq u$. 
Besides, the sequence $(u_n)_{n\in\N^*}$ strongly converges towards $u$ in $V$. 
\end{theorem}

The following Lemma will be used later. Although the proof is given in~\cite{Figueroa}, we recall it here for the sake of self-containedness.
\begin{lemma}\label{lem:maxminform}
For all $v\in V$, let us denote by $\|v\|_a:= \sqrt{a(v,v)}$. Then, for all $n\in\N^*$, 
\begin{equation}\label{eq:eqmaxmin}
\left \|r_n^{(1)}\otimes \cdots \otimes r_n^{(d)}\right\|_a = \mathop{\sup}_{(r^{(1)}, \cdots , r^{(d)}) \in V_{x_1} \times \cdots \times V_{x_d}, \; r^{(1)}\otimes \cdots \otimes r^{(d)} \neq 0}\frac{a\left( u-u_{n-1}, r^{(1)} \otimes \cdots \otimes r^{(d)}\right)}{\left\| r^{(1)} \otimes \cdots \otimes r^{(d)}\right\|_a}.
\end{equation}
\end{lemma}

\begin{proof}
 Let us prove (\ref{eq:eqmaxmin}) for $n=1$. The proof is similar for larger $n\in\N^*$. The $d$-tuple $\left(r_1^{(1)}, \cdots, r_1^{(d)}\right)\in V_{x_1}\times \cdots \times V_{x_d}$ 
solution of (\ref{eq:algomin}) for $n=1$ equivalently satisfies:
\begin{equation}\label{eq:2ndform}
\left( r_1^{(1)} , \cdots , r_1^{(d)} \right) \in \mathop{\mbox{\rm argmin}}_{\left( r^{(1)}, \cdots , r^{(d)} \right) \in V_{x_1} \times \cdots \times V_{x_d}}\frac{1}{2}\left\|u - r^{(1)} \otimes \cdots \otimes r^{(d)}\right\|_a^2. 
\end{equation}
The Euler equations associated to this minimization problem read: for all $\left( \delta r^{(1)}, \cdots, \delta r^{(d)}\right) \in V_{x_1} \times \cdots \times V_{x_d}$, 
\begin{eqnarray*}
&& a\left( r_1^{(1)} \otimes  \cdots \otimes r_1^{(d)} , r_1^{(1)} \otimes \cdots \otimes r_1^{(d-1)} \otimes \delta r^{(d)} + r_1^{(1)} \otimes \cdots \otimes r_1^{(d-2)} \otimes \delta r^{(d-1)} \otimes r_1^{(d)} + \cdots + \delta r^{(1)} \otimes r_1^{(2)} \otimes \cdots \otimes r_1^{(d)} \right)\\
&& =  a\left( u, r_1^{(1)} \otimes \cdots \otimes r_1^{(d-1)} \otimes \delta r^{(d)} + r_1^{(1)} \otimes \cdots \otimes r_1^{(d-2)} \otimes \delta r^{(d-1)} \otimes r_1^{(d)} + \cdots + \delta r^{(1)} \otimes r_1^{(2)} \otimes \cdots \otimes r_1^{(d)} \right),\\
\end{eqnarray*}
which implies that 
\begin{equation}\label{eq:normeq}
 \left\|r_1^{(1)}\otimes \cdots \otimes r_1^{(d)} \right\|_a^2 = a\left( u, r_1^{(1)}\otimes \cdots \otimes r_1^{(d)}\right).
\end{equation} 
Let now $\left( r^{(1)}, \cdots , r^{(d)} \right) \in V_{x_1} \times \cdots \times V_{x_d}$ be such that $r^{(1)} \otimes \cdots \otimes r^{(d)} \neq 0$. Using (\ref{eq:2ndform}) and 
(\ref{eq:normeq}), it holds that
$$
\left\| u -\frac{a\left( u, r_1^{(1)}\otimes \cdots \otimes r_1^{(d)} \right)}{\left\| r_1^{(1)} \otimes \cdots \otimes r_1^{(d)} \right\|_a^2} r_1^{(1)}\otimes \cdots \otimes r_1^{(d)} \right\|_a^2 = \left\| u - r_1^{(1)}\otimes \cdots \otimes r_1^{(d)} \right\|_a^2 \leq \left\| u - \frac{a\left( u, r^{(1)} \otimes \cdots \otimes r^{(d)}\right)}{\left\| r^{(1)} \otimes \cdots \otimes r^{(d)}\right\|_a^2}r^{(1)} \otimes \cdots \otimes r^{(d)}\right\|_a^2.
$$
Therefore, 
$$
\frac{a\left( u, r_1^{(1)}\otimes \cdots \otimes r_1^{(d)} \right)^2}{\left\| r_1^{(1)} \otimes \cdots \otimes r_1^{(d)} \right\|_a^2} \geq \frac{a\left( u, r^{(1)}\otimes \cdots \otimes r^{(d)} \right)^2}{\left\| r^{(1)} \otimes \cdots \otimes r^{(d)} \right\|_a^2}.
$$
Taking the supremum over all $\left( r^{(1)}, \cdots , r^{(d)} \right)\in V_{x_1} \times \cdots \times V_{x_d}$ such that $r^{(1)} \otimes \cdots \otimes r^{(d)}\neq 0$ yields the result.  
\end{proof}

Equation (\ref{eq:eqmaxmin}) implies in particular that for all $n\in\N^*$, 
\begin{equation}\label{eq:touse}
\left\| r_n^{(1)} \otimes \cdots \otimes \cdots r_n^{(d)} \right\|_a = \mathop{\sup}_{\left( r^{(1)}, \cdots , r^{(d)} \right) \in V_{x_1} \times \cdots \times V_{x_d} } \frac{l\left( r^{(1)} \otimes \cdots \otimes r^{(d)} \right) - a\left( u_{n-1}, r^{(1)} \otimes \cdots \otimes r^{(d)}\right)}{\left\| r^{(1)} \otimes \cdots \otimes r^{(d)} \right\|_a}.
\end{equation}

\vspace{1cm}

Let us rewrite the greedy algorithm in the particular case when $d=2$. 
\begin{enumerate}
 \item Let $u_0 = 0$ and $n=1$;
\item define $\left(r_n,s_n \right) \in V_x \times V_t$ such that 
\begin{equation}\label{eq:algomin2}
\left(r_n,s_n \right) \in \mathop{\mbox{\rm argmin}}_{\left( r, s \right) \in V_x \times V_t} \cE\left( u_{n-1} + r \otimes s \right);
\end{equation}
\item define $u_n = u_{n-1} + r_n \otimes s_n$ and set $n=n+1$. 
\end{enumerate}

For the sake of simplicity, in the rest of the article, all the algorithms will be presented in the case when $d=2$. The generalization of the approaches to a
 larger number 
of variates $d$ is straightforward unless mentioned. 

\medskip

The Euler equations associated to the minimization problem (\ref{eq:algomin2}) read
\begin{equation}\label{eq:Eulermin}
a(u_{n-1} + r_n\otimes s_n, \delta r \otimes s_n + r_n\otimes \delta s) = l(\delta r \otimes s_n + r_n \otimes \delta s), \; \forall (\delta r, \delta s)\in V_x \times V_t. 
\end{equation}

As a consequence of Theorem~\ref{th:sym}, provided that the set
\begin{equation}\label{eq:defSigma2}
\Sigma = \left\{ r\otimes s, \; r\in V_x, \; s\in V_t\right\}
\end{equation}
satisfies assumptions (A1) and (A2), at the first iteration of the algorithm ($n=1$), as soon as the form $l$ is nonzero, 
there exists at least one solution $(r_1, s_1)\in V_x\times V_t$ of
$$
a(r_1\otimes s_1, \delta r \otimes s_1 + r_1\otimes \delta s) = l(\delta r \otimes s_1 + r_1 \otimes \delta s), \; \forall (\delta r, \delta s)\in V_x \times V_t, 
$$ 
such that $r_1\otimes s_1\neq 0$.

\medskip

In practice, at each iteration $n\in\N^*$, a pair $(r_n,s_n)\in V_x \times V_t$ is computed via the resolution of the Euler equations (\ref{eq:Eulermin}) using a fixed-point procedure which 
reads as follows: 
\begin{itemize}
 \item choose $\left(r_n^{(0)}, s_n^{(0)}\right)\in V_x \times V_t$ and set $m=1$;
\item find $\left(r_n^{(m)}, s_n^{(m)}\right)\in V_x \times V_t$ such that
\begin{equation}\label{eq:fp}
\left\{
\begin{array}{l}
a\left(u_{n-1} + r^{(m)}_n\otimes s^{(m-1)}_n, \delta r \otimes s^{(m-1)}_n\right) = l\left(\delta r \otimes s^{(m-1)}_n \right), \; \forall \delta r \in V_x, \\  
a\left(u_{n-1} + r^{(m)}_n\otimes s^{(m)}_n, r^{(m)}_n\otimes \delta s\right) = l\left(r^{(m)}_n \otimes \delta s\right), \; \forall \delta s\in V_t;\\ 
\end{array}
\right. 
\end{equation}
\item set $m=m+1$.
\end{itemize}

This fixed-point algorithm is numerically observed to converge exponentially fast in most situations, although, at least to our knowledge, there is no rigorous proof in the general case. 

\section{The non-symmetric case} \label{sec:nonsym}

\subsection{General framework}

Let us now consider the case of a non-symmetric linear problem of the form
\begin{equation}\label{eq:nonsym}
\left\{
\begin{array}{l}
 \mbox{find }u\in V\mbox{ such that}\\
\forall v\in V, \; a(u,v) = l(v),\\
\end{array}
\right.
\end{equation}
where
\begin{itemize}
 \item $a(\cdot, \cdot)$ is a \bfseries nonsymmmetric \normalfont continuous bilinear form on $V\times V$;
\item $l$ is a continuous linear form on $V$.
\end{itemize}
In the rest of the article, we will assume that 
\begin{itemize}
 \item[(A3)] problem (\ref{eq:nonsym}) has a unique solution $u\in V$ for any continuous linear form $l\in \mathfrak{L}(V,\R)$. 
\end{itemize}

\medskip

We denote by $\cA\in \mathfrak{L}(V,V)$ the operator defined by
$$
\forall v,w\in V,\; \langle \cA v,w\rangle_V = a(v,w),
$$
and by $\mathcal{L}$ the element of $V$ such that
$$
\forall v\in V, \; \langle \cL, v \rangle_V = l(v).
$$

We also introduce the operator $A: V \to V'$ and the linear form $L\in V'$ defined by $A = R_V \cA$ and $L = R_V \cL$ so that the unique solution $u$ to (\ref{eq:nonsym})
 is also the unique solution to the problem
$$
\left\{ 
\begin{array}{l}
 \mbox{find }u\in V\mbox{ such that}\\
Au = L \mbox{ in }V'. \\
\end{array}
\right .
$$

It follows from assumption (A3) that $\cA$ and $A$ are invertible operators. 

\subsection{Prototypical examples}

Let us present two prototypical examples we will refer to throughout the rest of the paper. 
\begin{itemize}
 \item The first one is
\begin{equation}\label{eq:UQ}
\left\{
\begin{array}{l}
 \mbox{find } u \in H^1_0(\cX)\otimes L^2(\cT) \mbox{ such that}\\
- \Delta_x u + b_x\cdot \nabla_x u + u = f \mbox{ in }\cD'(\cX\times \cT),\\
\end{array}
\right .
\end{equation}
with $f\in H^{-1}(\cX)\otimes L^2(\cT)$ and $b_x \in \R^{m_x}$. 
For this problem, $V = H^1_0(\cX)\otimes L^2(\cT)$, $V' = H^{-1}(\cX) \otimes L^2(\cT)$ and
$$
\begin{array}{l}
 \forall u,v \in V,\;  a(u,v) =  \int_{\cX \times \cT} \left(\nabla_x u \cdot \nabla_x v + v (b_x\cdot \nabla_x u) + uv\right),\\
\forall v \in V, \; l(v) = \int_{\cT} \langle f, v\rangle_{H^{-1}(\cX), H^1_0(\cX)} .\\
\end{array}
$$
In this case, $A = -\Delta_x + b_x\cdot \nabla_x +1$. 

\item The second example is
\begin{equation}\label{eq:Poisson}
\left\{
\begin{array}{l}
 \mbox{find } u \in H^1_0(\cX \times \cT) \mbox{ such that}\\
- \Delta_{x,t} u + b\cdot \nabla_{x,t} u + u = f \mbox{ in }\cD'(\cX\times \cT),\\
\end{array}
\right .
\end{equation}
with $f\in H^{-1}(\cX\times \cT)$ and $b=(b_x, b_t) \in \R^{m_x} \times \R^{m_t}$. 
For this problem, $V =  H^1_0(\cX\times \cT)$, $V' = H^{-1}(\cX, \cT)$ and
$$
\begin{array}{l}
 \forall u,v \in V, \; a(u,v) = \int_{\cX \times \cT} \left(\nabla_{x,t} u \cdot \nabla_{x,t} v + v (b\cdot \nabla_{x,t} u) + uv\right),\\
\forall v \in V, \; l(v) = \langle f, v\rangle_{H^{-1}(\cX\times \cT), H^1_0(\cX\times \cT)}.\\
\end{array}
$$
In this case, $A = -\Delta_{x,t} + b\cdot \nabla_{x,t} +1$. 
\end{itemize}

\subsection{Failure of the standard greedy algorithm}\label{sec:galerkin}

Problem (\ref{eq:nonsym}) cannot be written as a minimization problem of the form (\ref{eq:minpb}) with an energy functional given by (\ref{eq:symener}). 
The definition of the greedy algorithm via the 
minimization problems (\ref{eq:algomin}) or (\ref{eq:algomin2}) cannot therefore be transposed to this case. However, a natural way to define the iterations of a greedy algorithm for the non-symmetric problem 
(\ref{eq:nonsym}) is to define iteratively for $n\in\N^*$ the pair $(r_n,s_n)\in V_x \times V_t$ as a solution of the following equation
\begin{equation}\label{eq:algnaive}
a(u_{n-1} + r_n\otimes s_n, \delta r \otimes s_n + r_n\otimes \delta s) = l(\delta r \otimes s_n + r_n \otimes \delta s), \; \forall (\delta r, \delta s)\in V_x \times V_t,  
\end{equation}
by analogy with the Euler equations (\ref{eq:Eulermin}). This is the so-called \itshape PGD-Galerkin \normalfont algorithm~\cite{Nouyprep}. 

\medskip

Actually, there are cases when $l\neq 0$ and any solution $(r_1,s_1) \in V_x \times V_t$ of the first iteration of the algorithm
\begin{equation}\label{eq:Eulerfirst}
a(r_1\otimes s_1, \delta r \otimes s_1 + r_1\otimes \delta s) = l(\delta r \otimes s_1 + r_1 \otimes \delta s), \; \forall (\delta r, \delta s)\in V_x \times V_t, 
\end{equation}
necessarily satisfies $r_1\otimes s_1 = 0$. Such an algorithm cannot converge since the approximation 
$\dps u_n = \sum_{k=1}^n r_k\otimes s_k$ given by the algorithm is equal to $0$ for any $n\in \N^*$. Besides, this situation may occur even 
when the norm of the antisymmetric part of the bilinear form $a(\cdot, \cdot)$ is arbitrarily small. 

\medskip

Let us give an explicit example. 

\begin{example}\label{ex:continuous}
 Let $\cX = \cT = (-1,1)$ and $\mu_x$ (respectively $\mu_t$) be the Lebesgue measure on $\cX$ (respectively on $\cT$). 
Let $b\in \R$, $V_x = H^1_{\rm per}(-1,1)$, $V_{t} = L^2(-1,1)$ and 
$V = V_x \otimes V_t$. 
Consider the non-symmetric problem (\ref{eq:nonsym}) with
$$
\forall v, w\in V, \; a(v,w) = \int_{\cX\times \cT} \left(\nabla_x v \cdot \nabla _x w + (b\cdot \nabla_x v) w + vw\right),
$$ 
and
$$
\forall v\in V, \; l(v) = \int_{\cX\times \cT} fv, 
$$
with $f\in  L^2_{\rm per}(-1,1) \otimes L^2(-1,1)$. 

\medskip

Problem (\ref{eq:nonsym}) is equivalent to
\begin{equation}
 \left\{
\begin{array}{l}
 \mbox{ find }u\in H^1_{\rm per}(-1,1) \otimes L^2(-1,1) \mbox{ such that }\\
- \Delta_x u + b\nabla_x u + u = f \quad \mbox{in }\cD'(\R\times \cT).\\
\end{array}
\right .
\end{equation}

In this context, equations (\ref{eq:Eulerfirst}) read
\begin{equation}\label{eq:ELnonsym}
 \left\{
\begin{array}{l}
\mbox{find }(r_1,s_1) \in H^1_{\rm per}(-1,1) \times L^2(-1,1)  \mbox{ such that }\\
 \left[\int_{-1}^1 |s_1(t)|^2\,dt\right] \left( - r_1''(x) + br_1'(x) + r_1(x)\right) = \int_{-1}^1 f(x,t) s_1(t)\,dt,\\
\left[ \int_{-1}^1 \left(|r_1'(x)|^2 + |r_1(x)|^2\right)\,dx \right] s_1(t) = \int_{-1}^1 f(x,t)r_1(x)\,dx,\\
\end{array}
\right .
\end{equation}
since the periodic boundary conditions on $r_1$ imply that $\int_{-1}^1 r_1(x)r_1'(x)\,dx = 0$. 

\medskip

Unlike the symmetric case, there exists an infinite set of functions $f\in L^2_{\rm per}(-1,1) \otimes L^2(-1,1)$ such that $f\neq 0$ and any solution $(r_1,s_1)\in V_x \times V_t$ of 
equations (\ref{eq:ELnonsym}) necessarily satisfies $r_1\otimes s_1 = 0$ for any arbitrarily small value of $|b|$. 
This is the case for example when $f(x,t) = \phi(x-t)$ for all $(x,t) \in \R\times (-1,1)$ with $\phi\in L^2_{\rm per}(-1,1)$ an odd real-valued function. 

Let us argue by contradiction. If $(r_1,s_1)\in V_x \times V_t$ is a solution to (\ref{eq:ELnonsym}) such that $r_1\otimes s_1\neq 0$, up to some rescaling, we can assume that 
$$
\int_{-1}^1 |s_1(t)|^2\,dt  = \int_{-1}^1 \left(|r_1'(x)|^2 + |r_1(x)|^2\right)\,dx  = \lambda >0.
$$
Thus, we can rewrite (\ref{eq:ELnonsym}) as
\begin{eqnarray*}
 - r_1''(x) + b r_1'(x)+ r_1(x) & = & \frac{1}{\lambda} \int_{-1}^1 f(x,t) s_1(t)\,dt,\\
s_1(t) & = & \frac{1}{\lambda} \int_{-1}^1 f(x,t)r_1(x)\,dx.\\
\end{eqnarray*}
Plugging the second equation into the first one, we obtain
\begin{equation}\label{eq:befFour}
 - r_1''(x) + b r_1'(x)+ r_1(x)  =  \frac{1}{\lambda^2} \int_{-1}^1 \left(\int_{-1}^1f(x,t) f(y,t)\,dt\right) r_1(y)\,dy.
\end{equation}
Let us denote by $g(x,y) = \int_{-1}^1f(x,t) f(y,t)\,dt$ for all $(x,y)\in \R^2$. As $\phi$ is an odd, $2$-periodic function, it holds that
\begin{eqnarray*}
 g(x,y) & =& \int_{-1}^1f(x,t) f(y,t)\,dt\\
& = & \int_{-1}^1 \phi(x-t)\phi(y-t)\,dt\\
& = & - \int_{-1}^1 \phi(x-t) \phi(t-y)\,dt\\
& = & -\int_{-1 +y}^{1+y} \phi(x-y-u) \phi(u)\,du\\
& = & -\int_{-1}^{1} \phi(x-y-u) \phi(u)\,du.\\
\end{eqnarray*}
Taking the Fourier transform of equation (\ref{eq:befFour}) yields that for all $k\in\pi\Z$, 
$$
(|k|^2 + ibk +1)\rh_1(k) = - \frac{4}{\lambda^2}\left(\phih(k)\right)^2\rh_1(k),
$$
where 
$$
\rh_1(k) = \frac{1}{2}\int_{-1}^1 r_1(x) e^{-ik\cdot x}\,dx.
$$
Futhermore, $\lambda \in \R_+^*$ and $\phih(0) = 0$ ($\phi$ is an odd function). Thus, since $\phih(k)$ is a purely imaginary number, $-\left(\phih(k)\right)^2 = \left|\phih(k)\right|^2$ and a solution $r_1$ necessarily 
satisfies $\rh_1(k) = 0$ for all $k\in\pi\Z$, which yields a contradiction. 
\end{example}

This example clearly shows that a naive transposition of the greedy algorithm to the non-symmetric case by analogy with the Euler equations (\ref{eq:Eulermin}) obtained in 
the symmetric case may be doomed to failure. 

\medskip

This article presents a review of some methods which aim at circumventing this difficulty. 
A particular highlight is set on the practical implementation of these methods and on the existence of theoretical rigorous convergence results. The 
properties of the different algorithms which are dealt with in this article are summarized in Figure~1. 

 \begin{figure}\label{fig:tab}
   \centering
   \input{./tabsum2.pstex_t}
   \caption{Summary of the different greedy algorithms used for non-symmetric high-dimensional linear problems. }
 \end{figure}

\section{Residual minimization algorithms}\label{sec:resmin}

In this section, we present some numerical methods used 
for the computation of separate variable representations of the solution of non-symmetric problems, for which there 
are rigorous convergence proofs. A natural idea is to symmetrize (\ref{eq:nonsym}) using a  
reformulation as a residual minimization problem in a well-chosen norm. These algorithms are also called 
\itshape Minimum Residual PGD\normalfont in the literature~\cite{Nouyprep}.

\subsection{Minimization of the residual in the $L^2(\cX \times \cT)$ norm}\label{sec:resmin1}

Let us assume that $L\in L^2(\cX \times \cT)$ and that there exists $D(A) \subset V$ a dense subdomain of $L^2(\cX\times \cT)$ such that $A(D(A)) \subset L^2(\cX\times \cT)$. 
The mapping $A: D(A) \to L^2(\cX \times \cT)$ defines a linear operator on $L^2(\cX\times \cT)$. Let us assume moreover that $A$ is a closed operator. This implies in particular that $D(A)$, endowed with the scalar product
$$
\forall v,w\in D(A), \; \langle v, w \rangle_{D(A)} = \langle v,w\rangle + \langle Av, Aw \rangle,
$$
is a Hilbert space. 

\medskip

A first approach, inspired by \cite{Falco1}, consists in applying a standard greedy algorithm on the energy functional
$$
\cE(v) = \|Av -L\|_{L^2(\cX\times \cT)}^2, \; \forall v\in D(A).
$$

\medskip 

Let us consider the case when
$$
A= \sum_{i=1}^p A_x^{(i)} \otimes A_t^{(i)}
$$
where for all $1\leq i \le p$, $A_x^{(i)}$ and $A_t^{(i)}$ are operators on $L^2(\cX)$ and $L^2(\cT)$ with domains $D\left(A_x^{(i)}\right)$ and $D\left(A_t^{(i)}\right)$ 
respectively. 
We denote by $D_x = \bigcap_{i=1}^p D\left(A_x^{(i)}\right)$ and $D_t = \bigcap_{i=1}^p D\left(A_{t}^{(i)}\right)$, and assume that 
$D_x$ and $D_t$ are dense subspaces of $L^2(\cX)$ and $L^2(\cT)$ respectively and are Hilbert spaces, when endowed with the scalar products
$$
\forall v,w\in D_x, \; \langle v,w\rangle_{D_x} = \langle v, w\rangle_{\cX} + \sum_{i=1}^p \left\langle A_x^{(i)} v, A_x^{(i)}w \right\rangle_{\cX},
$$
and
$$
\forall v,w\in D_t, \; \langle v,w\rangle_{D_t} = \langle v, w\rangle_{\cT} + \sum_{i=1}^p \left\langle A_t^{(i)} v, A_t^{(i)}w \right\rangle_{\cT}.
$$
The greedy algorithm reads: 
\begin{enumerate}
 \item let $u_0=0$ and set $n=1$;
\item define $(r_n, s_n)\in D_x  \times D_t$ such that
\begin{equation}\label{eq:minL2}
(r_n,s_n) \in \mathop{\mbox{\rm argmin}}_{(r,s)\in D_x \times D_t} \| A(u_{n-1} + r\otimes s) -L \|_{L^2(\cX\times \cT)}^2;
\end{equation}
\item set $u_n = u_{n-1} + r_n\otimes s_n$ and $n=n+1$.  
\end{enumerate}

Let us denote by $\Sigma^D := \{ r\otimes s, \; r\in D_x , \; s\in D_t\}$. From Theorem~\ref{th:sym}, provided that
\begin{itemize}
 \item[(B1)] $ \overline{\mbox{\rm Span}\Sigma^D}^{D(A)}  = D(A)$;
\item[(B2)] $\Sigma^D$ is weakly closed in $D(A)$;
\end{itemize}
the sequence $(u_n)_{n\in\N^*}$ strongly converges towards $u$ in $D(A)$. 

\medskip

In the case of problem (\ref{eq:UQ}), $A = A_x\otimes A_t$ with $A_x = -\Delta_x + b\cdot \nabla_x +1 $ and $A_t = 1$, $D(A) = \left(H^2(\cX) \cap H^1_0(\cX) \right)\otimes L^2(\cT)$, $D_x = D(A_x) = H^2(\cX)\cap H^1_0(\cX)$ and $D_t = D(A_t)= L^2(\cT)$.

For problem (\ref{eq:Poisson}), $A = A_x^{(1)} \otimes A_t^{(1)} + A_x^{(2)}\otimes A_t^{(2)}$ with 
$A_x^{(1)} =  - \Delta_x + b_x \cdot \nabla_x +1$, $A_t^{(1)} = 1$, $A_x^{(2)} = 1$ and $A_t^{(2)}  = -\Delta_t + b_t\cdot \nabla_t$, 
$D(A) = H^2(\cX\times \cT)\cap H^1_0(\cX\times \cT)$, $D_x = H^2(\cX) \cap H^1_0(\cX)$ and $D_t = H^2(\cT) \cap H^1_0(\cT)$. 

In both cases, assumptions (A1) and (A2) are satisfied.

\medskip

Actually, when $L$ is regular enough, i.e. if $L\in D(A^*)$, where $A^*$ denotes the adjoint of $A$ and $D(A^*)$ its domain, this method is equivalent to performing a standard greedy algorithm on the symmetric coercive problem
$$
A^* Au = A^* L.
$$
The Euler equations associated to the minimization problems (\ref{eq:minL2}) read
$$
\left\langle A(u_{n-1} + r_n\otimes s_n) -L, A(\delta r \otimes s_n + r_n \otimes \delta s) \right\rangle = 0, \; \forall (\delta r, \delta s)\in D_x \times D_t.
$$

\medskip

This method suffers from several drawbacks though. Firstly, the right-hand side $L$ needs more regularity than necessary for problem (\ref{eq:nonsym}) to be well-posed (we need $L\in L^2(\cX \times \cT)$ instead of $L\in V'$). 

Secondly, and more importantly, the conditioning of the associated discretized problems behaves badly since it scales quadratically with the conditioning of the original problem $Au=L$.

\subsection{Minimization of the residual in the dual norm}\label{sec:resmin2}

 In order to avoid the conditioning problems encountered when minimizing the residual
 in the $L^2(\cX\times \cT)$ norm, another method consists in performing a greedy algorithm on the energy functional
$$
\cE(v) = \|Av -L\|_{V'}^2 = \|R_V^{-1}(Av-L)\|_V^2, \; \forall v\in V.
$$
Here, the residual $Av-L$ is evaluated in the dual norm $\|\cdot\|_{V'}$. In this method, the right-hand side $L$ does not need to be more regular than $L\in V'$ and this approach
 is equivalent to performing the standard greedy algorithm on the symmetric coercive 
problem
$$
A^*(R_V)^{-1}Au = A^*(R_V)^{-1}L.
$$
The conditioning of the resulting problem scales linearly with the conditioning of the original $Au = L$ problem.

The algorithm reads: 
\begin{enumerate}
 \item let $u_0=0$ and $n=1$; 
\item let $(r_n, s_n) \in V_x \times V_t$ such that
\begin{equation}\label{eq:mindual}
(r_n,s_n) \in \mathop{\mbox{\rm argmin}}_{(r,s)\in V_x \times V_t} \| (R_V)^{-1} \left[A(u_{n-1} + r\otimes s) -L\right] \|_{V}^2;
\end{equation}
\item set $u_n = u_{n-1} + r_n\otimes s_n$ and $n=n+1$. 
\end{enumerate}

Provided that $\Sigma$ defined by (\ref{eq:defSigma2}) satisfies assumptions (A1) and (A2), we infer from Theorem~\ref{th:sym} that the sequence $(u_n)_{n\in\N}$ strongly 
converges to $u$ in $V$.  

The Euler equations associated with the minimization problems (\ref{eq:mindual}) read: for all $(\delta r, \delta s)\in V_x\times V_t$, 
$$
\left\langle R_V^{-1}\left[ A(u_{n-1} + r_n \otimes s_n) -L \right], R_V^{-1}\left[ A(\delta r \otimes s_n + r_n \otimes \delta s) \right] \right\rangle_V = 0,
$$
or equivalently,
$$
\left\langle  A(u_{n-1} + r_n \otimes s_n) -L , R_V^{-1}\left[ A(\delta r \otimes s_n + r_n \otimes \delta s) \right] \right\rangle_{V', V}= 0.
$$

\medskip

However, even if the conditioning problem of the previous method is avoided, this algorithm still requires the inversion of the operator $R_V$. 

In the case when $V = V_x\otimes V_t$, the dual space $V'$ satisfies $V' = V_x'\otimes V_t'$, so that the operator $R_V = R_{V_x} \otimes R_{V_t}$ is a tensorized operator and $R_V^{-1} = R_{V_x}^{-1} \otimes R_{V_t}^{-1}$. 
A prototypical example of this situation is given in problem (\ref{eq:UQ}), where we have $V'_x = H^{-1}(\cX)$, 
$V'_t = L^2(\cT)$, $R_{V_x}= - \Delta_x$, $R_{V_t} = 1$ and $R_V = R_{V_x} \otimes R_{V_t}$. Thus, $R_V^{-1}= (-\Delta_x)^{-1} \otimes 1$ and carrying out 
the above greedy algorithm requires the computation of several low-dimensional Poisson problems, which remains doable but increases the time and memory needs compared to a standard greedy algorithm in the symmetric coervive case where $b = b_x = 0$.

\medskip

The situation is even more intricate when $V \neq V_x \otimes V_t$, since the operator $R_V$ is not a tensorized operator in general. A prototypical example of this situation is problem (\ref{eq:Poisson}) where 
$V' = H^{-1}(\cX \times \cT)$, $R_V = -\Delta_{x,t}$ and $R_V^{-1}$ cannot be expanded as a finite sum of tensorized operators.
These intermediate symmetric coercive high-dimensional can be solved with a standard greedy algorithm presented in Section~\ref{sec:sym}, but this 
considerably increases the time needed to run a simulation.  
 
In this particular case, since $R_V = -\Delta_x \otimes 1 - 1 \otimes \Delta_t$ is the sum of two tensorized operators which commute with one another, we can use an 
approach described in~\cite{Hackbusch}. This method consists in using an \itshape approximate expansion \normalfont of the inverse of the Laplacian operator, constructed as follows.   
The function $h: x\in [x_0, +\infty) \mapsto \frac{1}{x}$
(where $x_0$ is a positive real number) can be approximated by a sum of exponential functions of the form
$$
\frac{1}{x}\approx \sum_{l=1}^N C_l e^{-c_l x},
$$
for some $N\in \N^*$ and where $(C_l)_{1\leq l \leq N}$ and $(c_l)_{1\leq l \leq N}$ a two sets of well-chosen real numbers, depending on $x_0$.
Provided that $x_0$ satisfies $x_0 < \min(1, \lambda_1^x, \lambda_1^t)$, where $\lambda_1^x$ (respectively $\lambda_1^t$) is the lowest eigenvalue of the operator $-\Delta_x$ 
on $H^1_0(\cX)$ with respect to the $L^2(\cX)$ scalar product (respectively the lowest eigenvalue of the operator $-\Delta_t$ on $H^1_0(\cT)$ with respect to the $L^2(\cT)$ 
scalar product), since both the operators $-\Delta_x \otimes 1$ and $-1\otimes \Delta_t$ commute, $R_V^{-1}$ can be approximated by
\begin{equation}\label{eq:Hack}
\begin{array}{lll}
 R_V^{-1} & \approx & \sum_{l=1}^N C_l e^{-c_l\left(-\Delta_x \otimes 1 - 1 \otimes \Delta_t\right)}\\
& = & \sum_{l=1}^N C_l e^{-c_l \Delta_x} \otimes e^{-c_l \Delta_t}.\\
\end{array}
\end{equation}
The computation of the expansion (\ref{eq:Hack}) only involves the computation of the exponential of small-dimensional operators. But of course, to have a reliable approximation of this operator, 
the number $N$ of terms in the above approximation may be very large. Besides, an explicit expansion is not always available for a general operator $R_V^{-1}$. 

\medskip

The algorithms presented in the following sections are attempts to find numerical methods which
\begin{itemize}
 \item avoid the conditioning problem inherent to the method described in Section~\ref{sec:resmin1};
\item avoid the use of inverse operators such as $R_V^{-1}$ in the approach using the dual norm. 
\end{itemize}

Of course, a natural idea would be to find a suitable norm to minimize the residual to avoid the conditioning and inversion problems. So far, no norms with such properties have been 
proposed.  

In Section~\ref{sec:dual}, 
we present algorithms already existing in the literature, namely those suggested by Anthony Nouy~\cite{NouyMinMax} and Alexe\"i Lozinski~\cite{Lozinski}. 
In Section~\ref{sec:us}, a new algorithm is proposed. The known partial convergence results for these methods are presented and the numerical implementation 
of the algorithms are detailed.

\section{Algorithms based on dual formulations}\label{sec:dual}

In this section, we present some classes of algorithms based on dual formulations 
of the non-symmetric problem (\ref{eq:nonsym}). 

\subsection{MiniMax algorithm}\label{sec:minimax}

A first algorithm based on a dual formulation of problem (\ref{eq:nonsym}) is the \itshape MiniMax algorithm \normalfont proposed by Nouy~\cite{NouyMinMax}. 

The algorithm reads as follows:
\begin{enumerate}
 \item let $u_0 = 0$ and $n=1$;
\item let  $(r_n, \rt_n, s_n, \st_n)\in V_x^2 \times V_t^2$ such that
\begin{equation}\label{eq:minmax}
(r_n, \rt_n, s_n, \st_n) \in \mbox{arg} \mathop{\mbox{max}}_{(\rt, \st)\in V_x\times V_t} \mathop{\mbox{min}}_{(r,s)\in V_x \times V_t} \cJ_n(r\otimes s, \rt\otimes \st),
\end{equation}
where for all $v, \vt \in V$, 
$$
\cJ_n(v,\vt) = \frac{1}{2}\|v\|_V^2 - a(u_{n-1} + v, \vt) + l(\vt);
$$
\item set $u_n = u_{n-1} + r_n\otimes s_n$ and $n=n+1$. 
\end{enumerate}

At each iteration $n\in\N^*$, the computation of a quadruplet $(r_n, \rt_n, s_n, \st_n)\in V_x^2\times V_t^2$ satisfying (\ref{eq:minmax}) is done by solving the 
stationarity equations
\begin{equation}\label{eq:stationarity}
 \left\{ \begin{array}{ll}
a(u_{n-1} + r_n\otimes s_n, \rt_n \otimes \delta \st + \delta \rt \otimes \st_n) = l(\rt_n \otimes \delta \st + \delta \rt \otimes \st_n), &  \quad \forall (\delta \rt, \delta \st)\in V_x \times V_t,\\
a(r_n\otimes \delta s + \delta r \otimes s_n, \rt_n \otimes \st_n) = \langle r_n\otimes \delta s + \delta r \otimes s_n, r_n\otimes s_n\rangle_V, & \quad \forall (\delta r,\delta s)\in V_x \times V_t.\\         
\end{array}
\right .
\end{equation}

In practice, for each $n\in\N^*$, these equations are solved through a fixed-point procedure where the pairs $(r_n,\rt_n)\in V_x^2$ and $(s_n,\st_n)\in V_t^2$ are computed 
iteratively. More precisely, the fixed-point algorithm reads:
\begin{itemize}
 \item set $m=0$, and choose an inital guess $\left(r_n^{(0)}, \rt_n^{(0)}, s_n^{(0)}, \st_n^{(0)}\right)\in V_x^2\times V_t^2$;
\item find $\left(r_n^{(m+1)}, \rt_n^{(m+1)}\right)\in V_x^2$ such that
$$
 \left\{ \begin{array}{ll}
 a\left(u_{n-1} + r_n^{(m+1)}\otimes s_n^{(m)}, \delta \rt \otimes \st_n^{(m)}\right) = l\left(\delta \rt \otimes \st_n^{(m)}\right), &  \quad \forall \delta \rt \in V_x,\\
a\left(\delta r \otimes s_n^{(m)}, \rt_n^{(m+1)} \otimes \st_n^{(m)}\right) = \left\langle \delta r \otimes s_n^{(m)}, r_n^{(m+1)}\otimes s_n^{(m)}\right\rangle_V, & \quad \forall \delta r \in V_x;\\         
\end{array}
\right .
$$
\item find $\left(s_n^{(m+1)}, \st_n^{(m+1)}\right)\in V_t^2$ such that
$$
 \left\{ \begin{array}{ll}
a\left(u_{n-1} + r^{(m+1)}_n\otimes s^{(m+1)}_n, \rt^{(m+1)}_n \otimes \delta \st\right) = l\left(\rt^{(m+1)}_n \otimes \delta \st\right), &  \quad \forall \delta \st\in V_t,\\
a\left(r_n^{(m+1)}\otimes \delta s, \rt^{(m+1)}_n \otimes \st^{(m+1)}_n\right) = \left\langle r^{(m+1)}_n\otimes \delta s, r^{(m+1)}_n\otimes s^{(m+1)}_n\right\rangle_V, & \quad \forall \delta s\in V_t;\\         
\end{array}
\right .
$$
\item set $m = m+1$.
\end{itemize}

In \cite{Nouyminimax}, it is proved that in the case when $a = a_x\otimes a_t$ where $a_x$ is a continuous bilinear form on $V_x\times V_x$ 
and $a_t$ a continuous bilinear form on $V_t\times V_t$ and 
$V = V_x\otimes V_t$, the algorithm converges. However, there is no convergence result in the full general case.

\subsection{Greedy algorithms for Banach spaces}\label{sec:banach}

Another family of dual greedy algorithms is inspired from the methods suggested by Temlyakov in \cite{Temlyakov} for Banach spaces
 and was proposed by Lozinski~\cite{Lozinski} in order to deal with the resolution of high-dimensional problems of the form (\ref{eq:nonsym}). 

\vspace{1cm}

\subsubsection{Greedy algorithms for general Banach spaces}

For the sake of simplicity, let us present two particular greedy algorithms proposed by Temlyakov in the context of Banach spaces, namely the \itshape X-Greedy \normalfont and the 
\itshape Dual Greedy \normalfont algorithms.

Let $(X, \|\cdot\|_X)$ be a reflexive Banach space and $\cD$ a dictionary of $X$, i.e. a subset of $X$ such that 
for all $g\in \cD$, $\|g\|_X = 1$ and $\overline{\mbox{\rm Span}(\cD)}^X = X$. Let us also denote by $X^*$ the dual space of $X$. 

Let $f\in X$. The aim of both the Dual Greedy and the X-Greedy algorithms is to give an approximation of $f$ as a linear combination of vectors of the dictionary $\cD$. 
These numerical methods are generalizations of the \itshape Pure Greedy \normalfont algorithm, which is defined for Hilbert spaces. When $X$ is a Hilbert 
space endowed with the scalar product $\langle \cdot, \cdot \rangle_X$, the Pure Greedy algorithm can be interpreted in two equivalent ways, namely: 

\medskip

\bfseries Pure Greedy algorithm (1): \normalfont
\begin{enumerate}
 \item let $f_0 = 0$, $r_0 = f$ and $n=1$;
\item let $g_n\in \cD$ and $\alpha_n\in\R$ such that (assuming existence)
$$
\|r_{n-1} - \alpha_n g_n\|_X = \mathop{\mbox{min}}_{g\in \cD, \; \alpha \in \R} \|r_{n-1} - \alpha g\|_X;
$$
\item let $f_n = f_{n-1} + \alpha_n g_n$, $r_n = r_{n-1} - \alpha_n g_n$ and $n=n+1$; 
\end{enumerate}

\medskip

and

\medskip 

\bfseries Pure Greedy algorithm (2): \normalfont
\begin{enumerate}
 \item let $f_0 = 0$, $r_0 = f$ and $n=1$;
\item let $g_n\in \cD$ such that (assuming existence)
$$
\langle r_{n-1}, g_n\rangle_X = \mathop{\mbox{max}}_{g\in \cD} \langle r_{n-1}, g \rangle_X;
$$
\item let $\alpha_n \in  \R$ such that
$$
\|r_{n-1} - \alpha_n g_n \|_X = \mathop{\mbox{min}}_{\alpha\in\R} \|r_{n-1} - \alpha g_n\|_X;
$$
\item let $f_n = f_{n-1} + \alpha_n g_n$, $r_n = r_{n-1} - \alpha_n g_n$ and $n=n+1$. 
\end{enumerate}

\medskip

When $X$ is a Hilbert space, the two versions of the Pure Greedy algorithm are equivalent, but this is not the case anymore as soon as $X$ is a general Banach space. 

\medskip

The \itshape X-Greedy \normalfont algorithm corresponds to the extension of the first version of the Pure Greedy algorithm:
\begin{enumerate}
 \item let $f_0 = 0$, $r_0 = f$ and $n=1$;
\item let $g_n\in \cD$ and $\alpha_n \in \R$ such that (assuming existence)
\begin{equation}\label{eq:XGreedy}
\|r_{n-1} - \alpha_n g_n\|_X = \mathop{\mbox{min}}_{g\in \cD, \; \alpha \in \R} \|r_{n-1} - \alpha g\|_X;
\end{equation}
\item let $f_n = f_{n-1} + \alpha_n g_n$, $r_n = r_{n-1} - \alpha_n g_n$ and $n=n+1$. 
\end{enumerate}

\medskip

The \itshape Dual Greedy \normalfont algorithm generalizes the second version of the Pure Greedy algorithm and is slightly more subtle. It is based on the notion of \itshape peak functional\normalfont. For any non-zero element $f\in X$, 
we say that $F_f\in X'$ is a peak functional for $f$ if $\|F_f\|_{X^*} = 1$ and $F_f(f) = \|f\|_X$. The \itshape Dual Greedy \normalfont algorithm reads: 
\begin{enumerate}
 \item let $f_0 = 0$, $r_0 = f$ and $n=1$;
\item let $F_{r_{n-1}}\in X^*$ be a peak functional for $r_{n-1}$ and let $g_n\in \cD$ such that (assuming existence)
\begin{equation}\label{eq:DualGreedy1}
g_n \in \mathop{\mbox{\rm argmax}}_{g\in \cD} F_{r_{n-1}}(g);
\end{equation}
\item let $\alpha_n \in \R$ such that
\begin{equation}\label{eq:DualGreedy2}
\alpha_n \in \mathop{\rm argmin}_{\alpha\in \R} \|r_{n-1} - \alpha g_n\|_X;
\end{equation}
\item let $f_n = f_{n-1} + \alpha_n g_n$, $r_n = r_{n-1} - \alpha_n g_n$ and $n=n+1$. 
\end{enumerate}

\medskip

Slightly modified versions (relaxed versions) of the X-Greedy and Dual Greedy algorithms are proved to converge in \cite{Temlyakov} provided that the space $X$ and 
the dictionary $\cD$ satisfy some additional assumptions, detailed below. 

\medskip

We define the modulus of smoothness of the Banach space $X$ by
$$
\forall \beta \in \R, \; \rho(\beta) := \mathop{\sup}_{\|x\|_X = \|y\|_X = 1} \left( \frac{1}{2}(\|x+\beta y\|_X + \|x-\beta y\|_X) -1 \right).
$$
The Banach space $(X, \|\cdot\|_X)$ is said to be uniformly smooth~\cite{Temlyakov} if 
$$
\mathop{\lim}_{\beta\to 0} \frac{\rho(\beta)}{\beta} = 0.
$$
Let us point out that if a Banach space $(X, \|\cdot\|_X)$ is uniformly smooth, then the mapping 
$G: x \in X\mapsto \|x\|_X$ is Fr\'echet-differentiable. 

\medskip

The relaxed versions of the X-Greedy and Dual Greedy algorithms are proved to converge~\cite{Temlyakov} provided that
\begin{itemize}
\item[(B1)] $\overline{\mbox{\rm Span}(\cD)}^{\|\cdot\|_X} = X$; 
\item[(B2)] $\R \cD$ is weakly closed in $X$;
\item[(B3)] $X$ is a uniformly smooth Banach space.
\end{itemize}
We do not write here these relaxed versions of the algorithms for the sake of brevity and refer to~\cite{Temlyakov}. 

\vspace{1cm}

\subsubsection{Special Banach spaces for non-symmetric high-dimensional problems}

Let us now present how these ideas were adapted by Lozinski to the case of high-dimensional non-symmetric problems. We begin here with the description of the particular 
Banach spaces involved. Let us assume in the rest of Section~\ref{sec:banach} that 
the operator $\cA^{-1}:V \to V$ is bounded.

\vspace{1cm}

\bfseries A Banach space with good theoretical properties but which cannot be used in practice \normalfont

\medskip

The space $V$ is now endowed with the following dual norm
$$
\forall v\in V, \; \|v\|_\cA = \sup_{w\in V, \; w\neq 0} \frac{a(v,w)}{\|w\|_V} = \|\cA v \|_V = \|A v\|_{V'}.
$$
Actually, since the linear operator $\cA$ is bounded on $V$, the space $(V, \|\cdot\|_\cA)$ is a reflexive Banach space whose dual space is $(V, \|\cdot \|_{(\cA^*)^{-1}})$ where 
$$
\forall v\in V, \; \|v\|_{(\cA^*)^{-1}} = \sup_{w\in V, \; w\neq 0} \frac{a(w,v)}{\|w\|_V} = \|(\cA^*)^{-1}v \|_V.
$$
Let us show that the Banach space $(V, \|\cdot\|_{\cA})$ and the dictionary
$$
\cD = \{ r\otimes s, \; r\in V_x, \; s\in V_t, \; \|r\otimes s \|_\cA = 1 \}
$$
satisfy assumptions (B1), (B2) and (B3). 

\medskip

Let us begin with the proof of (B1) and (B2). Since the set of tensor product functions
$$
\Sigma = \{ r\otimes s, \; r\in V_x, \; s\in V_t \} = \R \cD
$$
is assumed to be weakly closed in $(V, \|\cdot\|_V)$ and to satisfy $\overline{\mbox{\rm Span}(\Sigma)}^{(V, \|\cdot\|_V)} = V$ (assumptions (A1) and (A2)), (B1) and (B2) are direct consequences 
of the fact that $\cA$ and $\cA^{-1}$ belong to the space $\mathfrak{L}(V,V)$ (i.e. are bounded operators). For instance, $\mathfrak{L}\left( (V, \|\cdot\|_V), \R\right) = \mathfrak{L}\left( (V, \|\cdot\|_\cA), \R\right)$ since
$$
\forall l\in  \mathfrak{L}\left( (V, \|\cdot\|_V), \R\right),  \frac{1}{\|\cA \|_{\mathfrak{L}(V,V)}} \|l\|_{\mathfrak{L}\left( (V, \|\cdot\|_V), \R\right)} \leq\|l\|_{\mathfrak{L}\left( (V, \|\cdot\|_\cA), \R\right)} \leq \|\cA^{-1} \|_{\mathfrak{L}(V,V)} \|l\|_{\mathfrak{L}\left( (V, \|\cdot\|_V), \R\right)}.
$$

Let us now prove (B3). Since the operator $\cA$ is invertible, the modulus $\rho_\cA$ of smoothness of $(V, \|\cdot\|_\cA)$ is equal to the modulus of smoothness $\rho$ of $(V, \|\cdot\|_V)$. Indeed, for all $\beta\in\R$, 
\begin{align*}
 \rho_\cA(\beta) & = \mathop{\sup}_{v,w\in V, \; \|v\|_{\cA} = \|w\|_{\cA} = 1} \left( \frac{1}{2}(\| v + \beta w \|_{\cA} + \|v -\beta w \|_\cA) -1 \right)\\
& = \mathop{\sup}_{v,w\in V, \; \|\cA v\|_{V} = \|\cA w\|_{V} = 1} \left( \frac{1}{2}(\| \cA v + \beta \cA w \|_{V} + \|\cA v -\beta \cA w \|_V) -1 \right)\\
& = \mathop{\sup}_{v,w\in V, \; \|v\|_{V} = \| w\|_{V} = 1} \left( \frac{1}{2}(\| v + \beta w \|_{V} + \| v - \beta w \|_V) -1 \right)\\
& = \rho(\beta).\\
\end{align*}
Since $(V, \|\cdot\|_V)$ is a Hilbert space, 
$$
\frac{\rho_{\cA}(\beta)}{\beta} = \frac{\rho(\beta)}{\beta} \mathop{\longrightarrow}_{\beta\to 0} 0, 
$$
and $(V, \|\cdot\|_{\cA})$ is a uniformly smooth Banach space. 

\medskip

To implement the X-Greedy or Dual Greedy algorithms in practice in this context, one needs to compute the norm $\|\cdot\|_\cA$ (see (\ref{eq:XGreedy}) and (\ref{eq:DualGreedy2})). 
Since for all $v\in V$, $\|v\|_\cA = \|\cA v\|_V = \|R_V^{-1} A v\|_V$, this requires the resolution of several intermediate low- or high-dimensional problems to compute the inverse 
of the operator $R_V$. The same issues as those described in Section~\ref{sec:resmin2} have to be faced.  

\vspace{1cm}

\bfseries Another more practical Banach space \normalfont

\medskip

The idea of Lozinski is to replace this norm 
by a weaker one, easier to compute,
\begin{equation}\label{eq:defiA}
\forall v\in V, \; \|v\|_{i\cA} = \sup_{(r,s)\in V_x \times V_t, \; r\otimes s\neq 0} \frac{a(v,r\otimes s)}{\|r\otimes s\|_V}. 
\end{equation}
Actually, denoting by $\|\cdot\|_i$ the \itshape injective norm \normalfont on $V$ \cite{tensorbook}, defined by
$$
\forall v\in V, \; \|v\|_i = \sup_{(r,s)\in V_x \times V_t, \; r\otimes s\neq 0} \frac{\langle v,r\otimes s\rangle_V}{\|r\otimes s\|_V}, 
$$
it holds that for all $v\in V$, $\|v\|_{i\cA} = \| \cA v \|_i$. 
Reasoning as above, the Banach space $(V,\|\cdot\|_{i\cA})$ has exactly the same properties as $(V,\|\cdot\|_i)$. 

 \medskip

Since $\Sigma$ is weakly closed in $V$, and since for all $v\in V$, $\|v\|_i \leq \|v\|_V$, $\Sigma$ is also weakly closed in $(V, \|\cdot\|_i)$. But, 
in the full general case, the Banach space $(V,\|\cdot\|_i)$ and hence the Banach space $(V,\|\cdot\|_{i\cA})$ are not uniformly smooth. 
Actually, these spaces may not even be reflexive. 
Indeed, let us assume that $V = V_x\otimes V_t$ and that $\|\cdot \|_V$ is the associated cross-norm, in other words that for all $(r,s)\in V_x\times V_t$, $\|r\otimes s\|_V = \|r\|_{V_x} \|s\|_{V_t}$.  It holds that 
\cite{tensorbook} $(V,\|\cdot\|_i)$ is isomorphous to $\cK(V_x, V_t)$, the Banach space of compact operators from $V_x$ to $V_t$ endowed with the operator norm. Since $\cK(V_x,V_t)$ is not a reflexive space ($\cK(V_x, V_t)^* = \mathfrak{S}_1(V_x,V_t)$ and $\mathfrak{S}_1(V_x, V_t)^* = \mathfrak{L}(V_x, V_t)$ where $\mathfrak{S}_1(V_x,V_t)$ denotes the 
set of trace-class operators from $V_x$ to $V_t$), there is no guarantee of convergence of the relaxed versions of the X-Greedy or Dual Greedy algorithms presented above. 

\vspace{1cm}

\bfseries The finite-dimensional cross-norm case \normalfont

\medskip
However, in the case when $V_x$ and $V_t$ are finite-dimensional and $V = V_x \otimes V_t$, the spaces $\cK(V_x,V_t)$ and $\mathfrak{L}(V_x,V_t)$ are identical. 
The space $(V,\|\cdot\|_i)$ is then reflexive and uniformly smooth. 
Indeed, if $V_x = \R^{m_x}$ and $V_t = \R^{m_t}$, since $\|\cdot\|_{V_x}$ and $\|\cdot\|_{V_t}$ both derive from the scalar products $\langle \cdot, \cdot \rangle_{V_x}$ 
and $\langle \cdot, \cdot \rangle_{V_t}$, there exist two invertible matrices $P\in \R^{m_x\times m_x}$ and $Q\in \R^{m_t\times m_t}$ such that
$$
\begin{array}{l}
 \forall r\in \R^{m_x}, \; \|r\|_{V_x} = \|Pr\|_{F_{m_x}},\\
\forall s\in \R^{m_t}, \; \|s\|_{V_t}  = \|Qs\|_{F_{m_t}},\\
\end{array}
$$
where $\|\cdot\|_{F_{m_x}}$ and $\|\cdot\|_{F_{m_t}}$ denote respectively the Frobenius norms on $\R^{m_x}$ and $\R^{m_t}$. 
Thus, $(V, \|\cdot\|_i)$ is isometrically isomorphic to 
$\R^{m_x\times m_t}$ seen as $\cK(V_x, V_t)$, endowed with the norm
$$
\forall M\in \R^{m_x\times m_t}, \; \|M\|_i = \mathop{\sup}_{r\in \R^{m_x}, \; r\neq 0} \frac{\|QMr\|_{F_{m_t}}}{\|Pr\|_{F_{m_x}}} = \|QMP^{-1}\|_2,
$$
where 
$$
\forall M\in \R^{m_x\times m_t}, \; \|M\|_2 = \mathop{\sup}_{r\in \R^{m_x}, \; r\neq 0} \frac{\|Mr\|_{F_{m_t}}}{\|r\|_{F_{m_x}}}.
$$
Actually, $(\R^{m_x\times m_t}, \|\cdot\|_2)$ is a uniformly smooth Banach space \cite{refsmooth}. Thus, greedy algorithms for Banach spaces do converge in this setting. 

% However, the fact that the algorithm may not converge in the general case of infinite dimensional spaces leads us to think that conditioning problems are to be expected when 
% these algorithms will be carried out in practice.    

\vspace{1cm}

\subsubsection{Practical implementation of the algorithms}

 \medskip

\bfseries X-Greedy algorithm\normalfont

\medskip

The X-Greedy algorithm reads:
\begin{enumerate}
 \item let $u_0 = 0$ and $n=1$; 
\item find $(r_n,s_n)\in V_x \times V_t$ such that
\begin{equation}
(r_n,s_n) \in \mathop{\mbox{argmin}}_{(r,s)\in V_x\times V_t} \| u - u_{n-1} -r\otimes s\|_{i\cA};
\end{equation}
\item set $u_n = u_{n-1} + r_n\otimes s_n$ and $n=n+1$. 
\end{enumerate}
From the definition of the norm $\|\cdot\|_{i\cA}$ (see (\ref{eq:defiA})), the second step of the algorithm can be rewritten as
\begin{itemize}
 \item [(2)] find $(r_n,s_n)\in V_x \times V_t$ such that
\begin{equation}\label{eq:Xgreedy}
\begin{array}{lll}
(r_n,s_n) & \in & \mathop{\mbox{argmin}}_{(r,s)\in V_x\times V_t} \mathop{\sup}_{(\rt, \st) \in V_x \times V_t} \frac{l(\rt\otimes \st) - a(u_{n-1}+r\otimes s, \rt\otimes \st)}{\|\rt \otimes \st\|_V} \\
 & = & \mathop{\mbox{argmin}}_{(r,s)\in V_x\times V_t} \mathop{\sup}_{(\rt, \st) \in V_x \times V_t, \; \|\rt\otimes \st\|_V = 1} l(\rt\otimes \st) - a(u_{n-1}+r\otimes s, \rt\otimes \st).\\
\end{array}
\end{equation}
\end{itemize}

\medskip

From a practical point of view, at each iteration $n\in\N^*$, the functions $(r_n, \widetilde{r}_n,s_n, \st_n) \in V_x^2 \times V_t^2$ are obtained by solving the stationarity equations associated with (\ref{eq:Xgreedy}), namely by solving the following coupled problem
$$
\left\{
\begin{array}{l}
 \mbox{find }(r_n, \rt_n,s_n, \st_n) \in V_x^2 \times V_t^2 \mbox{ such that for all }(\delta r, \delta \rt, \delta s, \delta \st)\in V_x^2 \times V_t^2, \\
\langle \rt_n \otimes \st_n, \rt_n \otimes \delta \st+ \delta \rt \otimes \st_n  \rangle_V + a(u_{n-1} + r_n\otimes s_n, \rt_n \otimes \delta \st+ \delta \rt \otimes \st_n ) = l(\rt_n \otimes \delta \st+ \delta \rt \otimes \st_n ), \\
a(r_n \otimes \delta s + \delta r \otimes s_n , \rt_n \otimes \st_n) = 0.\\
\end{array}
\right .
$$

The X-Greedy algorithm has not been implemented in practice yet. Indeed, it is not clear how to compute a solution of the above stationarity equations since using a fixed-point algorithm 
procedure similar to the one presented in Section~\ref{sec:minimax} for the MiniMax algorithm would always lead to $\rt_n \otimes \st_n = 0$, due to the form of the second equation.  

\vspace{1 cm}

\bfseries Dual Greedy algorithm \normalfont

\medskip

Let us describe here how Lozinski adapted the Dual Greedy algorithm for the resolution of high-dimensional non-symmetric linear problems. 

A remaining issue concerns the construction of a peak functional $F_{r_{n-1}}$ for the residual $r_{n-1} = u -u_{n-1}$ which is used in the second step of the algorithm (\ref{eq:DualGreedy1}). 
Actually, the true peak functional is not computed but only an approximation of this functional by an optimal tensor product function in a sense which is made precise below. 

The adapted Dual Greedy algorithm reads: 
\begin{enumerate}
 \item set $u_0 = 0$ and $n=1$; 
\item (computation of an approximate peak functional for the residual $r_{n-1} = u - u_{n-1}$ with a tensor product function) find $(\rt_n, \st_n)\in V_x\times V_t$ such that
\begin{equation}\label{eq:step1dg}
(\rt_n, \st_n)\in \mathop{\mbox{argmax}}_{(\rt, \st)\in V_x\times V_t, \; \|\rt\otimes \st\|_V = 1} a(u-u_{n-1}, \rt\otimes \st) = \mathop{\mbox{argmax}}_{(\rt, \st)\in V_x\times V_t, \; \|\rt\otimes \st\|_V = 1} l(\rt\otimes \st) - a(u_{n-1}, \rt\otimes \st);
\end{equation}
\item (Step~2 of the Dual Greedy algorithm, see (\ref{eq:DualGreedy1})) find $(r_n,s_n)\in V_x\times V_t$ such that
\begin{equation}\label{eq:step2dg}
(r_n, s_n)\in \mathop{\mbox{argmax}}_{(r,s)\in V_x\times V_t, \; \|r\otimes s\|_{i\cA} = 1} a(r\otimes s, \rt_n\otimes \st_n);
\end{equation}
\item (Step~3 of the Dual Greedy algorithm, see (\ref{eq:DualGreedy2})) find $\alpha_n \in \R$ such that
\begin{equation}\label{eq:step3dg}
\alpha_n \in \mathop{\mbox{argmin}}_{\alpha \in \R} \|u - u_{n-1} - \alpha r_n\otimes s_n\|_{i\cA} = \mathop{\mbox{argmin}}_{\alpha \in \R} \mathop{\sup}_{(\rt,\st)\in V_x\times V_t, \; \|\rt\otimes \st\|_V = 1} l(\rt\otimes \st) - a(u_{n-1} + \alpha r_n \otimes s_n, \rt\otimes \st);
\end{equation}
\item set $u_n = u_{n-1} + \alpha_n r_n\otimes s_n$ and $n = n+1$.
\end{enumerate}
 
\medskip

The Euler-Lagrange equations associated with the minimization problem (\ref{eq:step1dg}) read: for all $(\delta \rt, \delta \st)\in V_x \times V_t$, 
$$
\lambda_n \langle \rt_n\otimes \st_n, \rt_n\otimes \delta \st + \delta \rt \otimes \st_n \rangle_V = l(\rt_n\otimes \delta \st + \delta \rt \otimes \st_n) -a(u_{n-1}, \rt_n\otimes \delta \st + \delta \rt \otimes \st_n),
$$
for some $\lambda_n \in \R$ satisfying $\lambda_n \|\rt_n\otimes \st_n\|_V^2 = \lambda_n  = l(\rt_n\otimes \st_n) - a(u_{n-1}, \rt_n\otimes \st_n)$. 

\medskip

The Euler-Lagrange equations associated to the minimization problem (\ref{eq:step2dg}) can be rewritten as follows: $(r_n,s_n)\in V_x\times V_t$ is solution of
$$
a(r_n\otimes \delta s + \delta r \otimes s_n, \rt_n \otimes \st_n) = \mu_n a(r_n\otimes \delta s + \delta r \otimes s_n, \rh_n\otimes \sh_n), \; \forall (\delta r, \delta s)\in V_x \times V_t,
$$
where $(\rh_n, \sh_n)\in V_x\times V_t$ is such that
\begin{equation}\label{eq:rhsh}
(\rh_n, \sh_n) \in \mathop{\mbox{argmax}}_{(\rh, \sh)\in V_x\times V_t, \; \|\rh\otimes \sh\|_V = 1} a(r_n\otimes s_n, \rh\otimes \sh),
\end{equation}
and $\mu_n = a(r_n\otimes s_n, \rt_n \otimes \st_n)$. Besides, if the pair $(\rh_n, \sh_n)$ satisfies (\ref{eq:rhsh}), it holds that
$$
a(r_n\otimes s_n, \rh_n\otimes \delta \sh + \delta \rh \otimes \sh_n) = \nu_n \langle \rh_n\otimes \sh_n, \rh_n \otimes \delta \sh + \delta \rh\otimes \sh_n\rangle_V, \; \forall (\delta \rh, \delta \sh)\in V_x\times V_t,
$$
with $\nu_n = a(r_n\otimes s_n, \rh_n\otimes \sh_n) = \|r_n\otimes s_n\|_{i\cA} = 1$. 
This yields to a coupled problem on $(r_n,s_n)$ and $(\rh_n, \sh_n)$. 
Lozinski the noticed that, if $(r_n,s_n)$ is solution to (\ref{eq:step2dg}), $(\rh_n, \sh_n) = (\rt_n,\st_n)$ is solution of (\ref{eq:rhsh}) in the sense that 
$$
(\rt_n, \st_n) \in \mathop{\mbox{argmax}}_{(\rh, \sh)\in V_x\times V_t, \; \|\rh\otimes \sh\|_V = 1} a(r_n\otimes s_n, \rh\otimes \sh).
$$
The Euler-Lagrange equations can be rewritten as
$$
a(r_n\otimes s_n, \rt_n\otimes \delta \st + \delta \rt\otimes \st_n) = \langle \rt_n \otimes \st_n, \rt_n \otimes \delta \st + \delta \rt \otimes \st_n \rangle_V, \; \forall (\delta \rt, \delta \st)\in V_x\times V_t,
$$
by taking $\mu_n = a(r_n\otimes s_n, \rt_n\otimes \st_n) = \|r_n\otimes s_n\|_{i\cA} = 1$. 

\medskip

The Euler equations associated with (\ref{eq:step3dg}) read
$$
a(u-u_{n-1} - \alpha_n r_n\otimes s_n, \rt_n\otimes \st_n) = 0,
$$
yielding to $\alpha_n = a(u -u_{n-1},\rt_n\otimes \st_n) = \lambda_n$.

\medskip

Finally, replacing $\alpha_n r_n\otimes s_n$ by $r_n\otimes s_n$, the iterations of the Dual Greedy algorithm are computed in practice as follows:

\begin{enumerate}
 \item let $u_0 = 0$ and $n=1$;
\item find $(r_n, \rt_n , s_n , \st_n)\in V_x^2 \times V_t^2$ such that for all $(\delta r, \delta \rt, \delta s , \delta \st) \in V_x^2 \times V_t^2$,
$$
\left\{
\begin{array}{l}
 \langle \rt_n \otimes \st_n,  \rt_n \otimes \delta \st + \delta \rt \otimes \st_n\rangle_V = l(\rt_n \otimes \delta \st + \delta \rt \otimes \st_n) - a(u_{n-1} , \rt_n \otimes \delta \st + \delta \rt \otimes \st_n),\\
a(r_n\otimes s_n, \rt_n \otimes \delta \st + \delta \rt \otimes \st_n) = l(\rt_n \otimes \delta \st + \delta \rt \otimes \st_n) - a(u_{n-1} , \rt_n \otimes \delta \st + \delta \rt \otimes \st_n);\\
\end{array}
\right .
$$
\item set $u_n = u_{n-1} + r_n\otimes s_n$ and $n=n+1$. 
\end{enumerate}
These equations lead to two decoupled problems on $(r_n,s_n)$ and 
$(\rt_n, \st_n)$. Each of them is solved through a fixed-point procedure similar to (\ref{eq:fp}) described in details in Section~\ref{sec:fpalgos}. 
Some numerical tests are presented in Section~\ref{sec:num}, which illustrate the convergence of this algorithm.

\section{The Decomposition algorithm}\label{sec:us}

Let us now present a new algorithm based on a \itshape decomposition \normalfont of the bilinear form $a$.
The bilinear form $a$ can always be written as 
\begin{equation}\label{eq:decomposition}
 a(\cdot, \cdot) = b_s(\cdot, \cdot) + b(\cdot, \cdot)
\end{equation}
where $b_s(\cdot, \cdot)$ is a symmetric coercive continuous bilinear form on 
$V\times V$ and $b(\cdot, \cdot)$ a (not necessarily symmetric) continuous bilinear form on $V\times V$. In the sequel, $b_s(\cdot, \cdot)$ 
(respectively $b(\cdot, \cdot)$) will be refered to as 
the \itshape implicit \normalfont (respectively \itshape explicit\normalfont) 
part of $a(\cdot, \cdot)$. This terminology will be explained below. 

\begin{example}
 Let us introduce $a_s$ (respectively $a_{as}$) the symmetric part (respectively antisymmetric part) of $a(\cdot, \cdot)$, defined by:
\begin{equation}\label{eq:sympartdef}
\forall v,w\in V, \; a_s(v,w) = \frac{1}{2}\left( a(v,w) + a(w,v) \right),
\end{equation}
and
\begin{equation}\label{eq:antisympartdef}
\forall v,w\in V, \; a_{as}(v,w)  = \frac{1}{2}\left( a(v,w) - a(w,v) \right).
\end{equation}
Provided that $a_s$ is coercive, the decomposition $a(\cdot , \cdot) = a_s(\cdot, \cdot) + a_{as}(\cdot, \cdot)$ is admissible in the sense of (\ref{eq:decomposition}). 
\end{example}

\medskip

Let us denote by $\cB_s$ and $\cB$ the bounded operators on $V$ defined by
$$
\begin{array}{l}
\forall v,w\in V, \; b_s(v,w) = \langle \cB_s v,w\rangle_V ,\\
\forall v,w\in V, \; b(v,w) = \langle \cB v,w\rangle_V.\\
\end{array}
$$
Since $b_s$ is coercive, the operator $\cB_s$ is invertible. 

\medskip
  
The principle of the algorithm we consider in this section to solve problem (\ref{eq:nonsym}) consists in \itshape expliciting \normalfont the part $b$ of the bilinear form as a right-hand side source term. 
More precisely, one can consider the following fixed-point algorithm:
\begin{enumerate}
 \item choose a starting guess $u_0\in V$ and set $n=1$;
\item let $u_n$ be the unique solution to
\begin{equation}\label{eq:fixpoint}
 \forall v\in V, \; b_s(u_n,v) = l(v) - b(u_{n-1},v);
\end{equation}
\item set $n=n+1$. 
\end{enumerate}
In other words, $u_n = F(u_{n-1})$ where for all $v\in V$, $F(v) = \cB_s^{-1}(\cL - \cB v)$. If $\|\cB_s^{-1}\cB\|_{\mathfrak{L}(V,V)} < 1$, the mapping $F:V\to V$ is a contraction  
and it follows from the Picard fixed-point theorem that the sequence $(u_n)_{n\in\N}$ strongly converges in $V$ towards the solution $u$ of the initial problem~(\ref{eq:nonsym}). 

A natural approach thus consists in solving problem (\ref{eq:fixpoint}) at each iteration $n\in\N^*$ using a standard greedy procedure. Provided that the greedy expansion obtained at each iteration 
$n\in\N^*$ is accurate enough, the sequence $(u_n)_{n\in\N}$ given by this algorithm strongly converges in $V$ towards $u$.

\medskip

However, the principle of this method requires to compute a full greedy loop at each iteration $n\in\N^*$. In order to save computational time, we now introduce the following algorithm, 
in which only one tensor 
product function is computed at each iteration 
$n\in\N^*$.

\medskip

\bfseries Decomposition algorithm: \normalfont

\begin{enumerate}
 \item let $u_0 = 0$ and $n=1$;
\item find $(r_n, s_n)\in V_x \times V_t$ such that
\begin{equation}\label{eq:explicitalg}
(r_n,s_n) \in \mathop{\mbox{\rm argmin}}_{(r,s)\in V_x \times V_t} \frac{1}{2}b_s(u_{n-1} + r\otimes s, u_{n-1} + r\otimes s) - l(r\otimes s) -b(u_{n-1}, r\otimes s);  
\end{equation}
\item set $u_n = u_{n-1} + r_n \otimes s_n$ and set $n=n+1$.
\end{enumerate}
The bilinear form $b(\cdot, \cdot)$ is \itshape explicited \normalfont as a right-hand side, whereas $b_s(\cdot, \cdot)$ remains \itshape implicit\normalfont. This justifies the 
terminology introduced in the beginning of the section. 

\medskip
Equation (\ref{eq:explicitalg}) can be rewritten equivalently as
$$
\begin{array}{lll}
(r_n,s_n) & \in  & \mathop{\mbox{\rm argmin}}_{(r,s)\in V_x \times V_t} \frac{1}{2}b_s(r\otimes s, r\otimes s) - l(r\otimes s) -b_s(u_{n-1} + r\otimes s) -b(u_{n-1}, r\otimes s)\\
& = & \mathop{\mbox{\rm argmin}}_{(r,s)\in V_x \times V_t} \frac{1}{2}b_s(r\otimes s, r\otimes s) - l(r\otimes s)-a(u_{n-1}, r\otimes s).  \\
\end{array}
$$
This means that for each $n\in\N^*$, $r_n\otimes s_n$ is a tensor product solution to the first iteration of the greedy algorithm applied to 
the following symmetric coercive problem
$$
\left\{
\begin{array}{l}
 \mbox{find }u\in V\mbox{ such that}\\
\forall v\in V, \; b_s(u,v) = l(v) -b_s(u_{n-1},v) - b(u_{n-1}, v) = l(v) - a(u_{n-1}, v).\\
\end{array}
\right .
$$

\medskip

As explained above, such an algorithm is expected to converge only if the norm $\|\cB_s^{-1}\cB\|_{\mathfrak{L}(V,V)}$ is small enough.
Actually, in the case when the spaces $V_x$ and $V_t$ are finite dimensional, the following result holds: 
\begin{prop}\label{prop:explicit}
 If $V_x$ and $V_t$ are finite dimensional, there exists $\kappa >0$ such that if 
\begin{equation}\label{eq:rate}
\|\cB_s^{-1}\cB\|_{\mathfrak{L}(V,V)} \leq \kappa,
\end{equation}
then the sequence $(u_n)_{n\in\N^*}$ defined by the Decomposition algorithm strongly converges to $u$ in $V$. 
\end{prop}

\begin{proof}
Since $V_t$ and $V_x$ are assumed to be finite dimensional, from assumption (A1), so is $V$. Let $\kappa:=\|\cB_s^{-1} \cB\|_{\mathfrak{L}(V,V)}$.
Let us denote by $\langle \cdot, \cdot \rangle_{\Vt} = b_s(\cdot, \cdot)$ the scalar product on $V$ induced by the symmetric bilinear form 
$b_s$, and by $\|\cdot\|_{\Vt}$ the associated norm. 

Let $\cLt\in V$ and $\cBt \in \mathfrak{L}(V,V)$ be defined by
\begin{eqnarray*}
\forall v\in V, & \;  l(v) = \langle \cLt, v \rangle_{\Vt},\\
\forall v,w\in V, & \; b(v,w) = \langle \cBt v, w\rangle_{\Vt}.\\
\end{eqnarray*}
Actually, $\cBt = \cB_s^{-1}\cB$ and $\cLt = \cB_s^{-1} \cL$. This implies that for all $v\in V$, $\|\cAt v \|_{\Vt} \leq \kappa \|v\|_{\Vt}$.
Besides, $\cI + \cBt = \cB_s^{-1} (\cB_s + \cB) = \cB_s^{-1} \cA$, where $\cI$ denotes the identity operator on $V$.

\medskip

The proof relies on the fact that all norms are equivalent in finite dimension. In particular, let us define the injective norm $\|\cdot\|_{\it}$ by
$$
\forall v\in V, \; \|v\|_{\it} = \mathop{\sup}_{(r,s)\in V_x\times V_t} \frac{\langle v, r\otimes s\rangle_{\Vt}}{\|r\otimes s\|_{\Vt}}.
$$
Then, there exists $\alpha>0$ such that 
$$
\forall v\in V, \; \|v\|_{\Vt} \leq \alpha \|v\|_{\it}.
$$
Besides, for all $v\in V$, $ \|v\|_{\it} \leq \|v\|_{\Vt}$. 

\medskip

Let us introduce $\cUt_n:= \cLt-\cBt u_n - u_n = \cB_s^{-1}(\cL - \cA u_n)$. For all $n\in\N^*$, $\cUt_n$ is the vector of $V$ such that
$$
\forall v\in V, \; \langle \cUt_n, v \rangle_{\Vt} = l(v) - a(u_n, v).
$$
For $n\geq 1$, the Euler equations associated to (\ref{eq:explicitalg}) read:
$$
b_s(u_{n-1} + r_n\otimes s_n, r_n\otimes \delta s + \delta r \otimes s_n) = l(r_n\otimes \delta s + \delta r \otimes s_n) -b(u_{n-1}, r_n\otimes \delta s + \delta r \otimes s_n), \; \forall (\delta r, \delta s)\in V_x\times V_t.
$$
As a consequence, $l(r_n\otimes s_n) -a(u_{n-1}, r_n\otimes s_n) - b_s(r_n\otimes s_n) = \langle \cUt_{n-1} -r_n\otimes s_n, r_n\otimes s_n\rangle_{\Vt}  = 0$, which implies that
$\|\cUt_{n-1}\|_{\Vt}^2 = \|\cUt_{n-1}- r_n\otimes s_n\|_{\Vt}^2 + \|r_n\otimes s_n\|_{\Vt}^2$. 
Furthermore, from Lemma~\ref{lem:maxminform}, we have
$$
\|r_n\otimes s_n\|_{\Vt}= \mathop{\sup}_{(r,s)\in V_x \times V_t, \; r\otimes s\neq 0} \frac{\langle \cUt_{n-1}, r\otimes s\rangle_{\Vt}}{\|r\otimes s\|_{\Vt}} = \|\cUt_{n-1}\|_{\it}.
$$

\medskip

It holds that
\begin{eqnarray*}
 \|\cUt_{n-1}\|_{\Vt}^2 - \|\cUt_n\|_{\Vt}^2 & = & \|\cUt_{n-1}\|_{\Vt}^2 - \|\cUt_{n-1} - r_n\otimes s_n - \cBt r_n\otimes s_n\|_{\Vt}^2\\
& = & \|\cUt_{n-1}\|_{\Vt}^2 - \|\cUt_{n-1} -r_n\otimes s_n\|_{\Vt}^2 - \|\cBt r_n\otimes s_n\|_{\Vt}^2 \\
&& + 2\langle \cBt r_n\otimes s_n, \cUt_{n-1}-r_n\otimes s_n\rangle_{\Vt}\\
& = & \|r_n\otimes s_n\|_{\Vt}^2 - \|\cBt r_n\otimes s_n\|_{\Vt}^2 -2\langle \cBt r_n\otimes s_n, \cUt_{n-1}-r_n\otimes s_n\rangle_{\Vt}\\
&\geq & (1-\kappa^2) \|r_n\otimes s_n\|_{\Vt}^2 - 2 \kappa \|r_n\otimes s_n\|_{\Vt}\|\cUt_{n-1} -r_n\otimes s_n\|_{\Vt}\\
&\geq &(1-\kappa^2) \|r_n\otimes s_n\|_{\Vt}^2 - 2 \kappa \|r_n\otimes s_n\|_{\Vt}\|\cUt_{n-1}\|_{\Vt}\\
&\geq &  (1-\kappa^2) \|r_n\otimes s_n\|_{\Vt}^2 - 2 \kappa\alpha \|r_n\otimes s_n\|_{\Vt}\|\cUt_{n-1}\|_{\it}\\
&=&  (1-\kappa^2 - 2\alpha \kappa ) \|r_n\otimes s_n\|_{\Vt}^2\\
&=&  (1-\kappa^2 - 2\alpha \kappa ) \|\cUt_{n-1}\|_{\it}^2.\\
\end{eqnarray*}
If $\kappa$ is small enough to ensure that $1-\kappa^2 - 2\alpha \kappa > 0$, the sequence $(\|\cUt_n\|_{\Vt})_{n\in\N^*}$ is non-increasing, hence convergent. 
Thus, the series of general terms $(\|\cUt_{n-1}\|_{\Vt}^2 - \|\cUt_n\|_{\Vt}^2)_{n\in\N^*}$ and $(\|\cUt_n\|_{\it}^2)_{n\in\N^*}$ are convergent. This yields
$$
\|\cUt_n\|_{\it} \mathop{\longrightarrow}_{n\to\infty} 0,
$$
and, as all norms are equivalent in finite dimension,
$$
\|\cUt_n\|_{\Vt} \mathop{\longrightarrow}_{n\to\infty} 0.
$$
If $\kappa<1$, the operator $(\cI + \cBt)^{-1}$ 
is continuous from $V$ to $V$ for the norm $\|\cdot\|_{\Vt}$, and 
$$
\|u-u_n\|_{\Vt} = \|\cA^{-1}\cL -u_n\|_{\Vt} = \left\| \left( \cB_s^{-1} \cA\right)^{-1} \cB_s^{-1}\cL -u_n\right\|_{\Vt} = \|(\cI+\cBt)^{-1}(\cLt - (\cI+\cBt)u_n)\|_{\Vt}  = \|(\cI+\cBt)^{-1}\cUt_n\|_{\Vt} \mathop{\longrightarrow}_{n\to\infty} 0.
$$
Since the norm $\|\cdot\|_{\Vt}$ is equivalent to the norm $\|\cdot\|_V$, we obtain the desired result. 
\end{proof}

Unfortunately, the rate $\kappa$ in (\ref{eq:rate}) strongly depends on the dimensions of $V_x$ and $V_t$, as shown in the proof of Proposition~\ref{prop:explicit}. However, in some
 simple numerical experiments we performed with this algorithm, the rate $\kappa$ does not seem to depend on the dimension, as illustrated in Section~\ref{sec:num}. The analysis 
of this numerical observation is work in progress.

\medskip

A possible way to obtain a decomposition such as (\ref{eq:decomposition})  satisfying (\ref{eq:rate}) is to consider preconditioners for the initial antisymmetric 
problem. Problem (\ref{eq:nonsym}) can be rewritten as
\begin{equation}\label{eq:statit}
 \left\{
\begin{array}{l}
\mbox{find }u\in V\mbox{ such that}\\
u = (\cI - \cC \cA)u + \cC \cL,\\
\end{array}
\right .
\end{equation}
where $\cC$ is a well-chosen continuous linear operator on $V$ such that $\|\cI - \cC\cA\|_{\mathfrak{L}(V,V)}$ is as small as possible. With this formulation, 
one can consider the following algorithm: 
\begin{enumerate}
 \item let $u_0 = 0$ and $n=1$;
\item find $(r_n, s_n)\in V_x \times V_t$ such that
\begin{equation}
(r_n,s_n) \in \mathop{\mbox{\rm argmin}}_{(r,s)\in V_x \times V_t} \frac{1}{2}\left\langle u_{n-1} + r\otimes s, u_{n-1} + r\otimes s\right\rangle_V  - \langle (\cI - \cC\cA)u_{n-1} + \cC \cL, r\otimes s\rangle_V; 
\end{equation}
\item set $u_n = u_{n-1} + r_n\otimes s_n$ and $n = n+1$.
\end{enumerate}
The bilinear form associated with the continuous operator $\cC\cA$ needs to be easy to compute on tensor product functions. At least in 
the case when $V$ is finite dimensional, if $\|\cI -\cC\cA\|_{\mathfrak{L}(V,V)}$ is small enough, the sequence $(u_n)_{n\in\N^*}$ converges strongly in $V$ towards the solution 
$u$ of (\ref{eq:nonsym}). However, finding a suitable operator $\cC$ is not an easy task in general.

\section{Numerical results}\label{sec:num}

\subsection{Presentation of the toy problems}\label{sec:toy}

In this section, we present the two toy problems we consider in these numerical tests.  Let $b_x, b_t\in \R$, $m_x = m_t = 1$ and $\cX = \cT = (-1,1)$.

\medskip

The first toy problem is inspired from Example~\ref{ex:continuous}. In the continuous setting, we define 
\begin{itemize}
 \item $V := L^2_{\rm per}((-1,1), H^1_{\rm per}((-1,1), \C)) = L^2_{\rm per}((-1,1), \C)\otimes H^1_{\rm per}((-1,1),\C)$;
\item $V_x := H^1_{\rm per}((-1,1), \C)$;
\item $V_t := L^2_{\rm per}((-1,1), \C)$;  
\item $f\in L^2_{\rm per}((-1,1)^2, \C)$;
\item for all $u,v\in V$, 
$$
a(u,v) := \int_{-1}^1 \int_{-1}^1 \left( \nabla_x u \cdot \overline{\nabla_x v} + (b_x\cdot \nabla_x u)\overline{v}  + u \overline{v} \right),
$$ 
and 
$$
l(v):= \int_{-1}^1\int_{-1}^1 f\overline{v}.
$$
\end{itemize}

\medskip

The second toy problem is the following
\begin{itemize}
 \item $V := H^1_{\rm per}((-1,1)^2, \C))$;
\item $V_x := H^1_{\rm per}((-1,1), \C)$;
\item $V_t := H^1_{\rm per}((-1,1), \C)$;  
\item $f\in L^2_{\rm per}((-1,1)^2, \C)$;
\item for all $u,v\in V$, 
\begin{eqnarray*}
a(u,v) & := & \int_{-1}^1 \int_{-1}^1 \left( \nabla_x u \cdot \overline{\nabla_x v} + \nabla_t u\cdot \overline{\nabla_t v}\right) \\
&& + \int_{-1}^1 \int_{-1}^1\left( (b_x\cdot\nabla_x u +  b_t\cdot \nabla_t u)\overline{v}  + u\overline{v} \right),\\
\end{eqnarray*}
and
$$
l(v):=\int_{-1}^1\int_{-1}^1 f\overline{v}.
$$
\end{itemize}

It can be easily checked in the two cases that the Hilbert spaces $V$, $V_x$ and $V_t$ satisfy assumptions (A1) and (A2). Besides, the sesquilinear form $a: V\times V \to \C$ is continuous and satisfies 
assumption (A3). Our aim is to approximate the function $u\in V$ such that 
\begin{equation}\label{eq:testpbc}
\forall v\in V, \; a(u,v) = l(v).
\end{equation}

We introduce finite-dimensional discretization spaces defined as follows. For all $k\in \Z$, let $e_k^x: \cX \ni x \mapsto \frac{1}{\sqrt{2}}e^{ikx}$ and 
$e_k^t: \cT\ni t \mapsto \frac{1}{\sqrt{2}}e^{ikt}$. For $N_x, N_t \in \N^*$, we define
$$
V_x^{N_x}:=\mbox{\rm Span}\left\{ e_k^x, \; -N_x \leq k \leq N_x \right\},
$$
$$
V_t^{N_t}:=\mbox{\rm Span}\left\{ e_k^t, \; -N_t \leq k \leq N_t \right\},
$$
and $V^{N_x, N_t}:= V_x^{N_x} \otimes V_t^{N_t}$. 

We then consider the Galerkin approximation of problem (\ref{eq:testpbc}) in the discretization space $V^{N_x, N_t}$, i.e. find $u\in V^{N_x, N_t}$ such that
\begin{equation}\label{eq:pbd}
\forall v \in V^{N_x,N_t}, \; a( u, v) = l(v). 
\end{equation}
% For all $v\in V^{N_x,N_t}$, let us denote by $W = (w_{kl})_{-N_x\leq k \leq N_x, \; -N_t\leq l \leq N_t} \in \C^{(2N_x+1)\times(2N_t+1)}$ such that
% $$
% v = \sum_{-N_x \leq k \leq N_x, \; -N_t \leq l \leq N_t} w_{kl} e_k^x \otimes e_l^t.
% $$
% From now on, we consider that the continuous form $l$ is such that
% \begin{equation}\label{eq:eql}
% \forall v\in V^{N_x, N_t}, \; l(v) = \mbox{\rm Tr}(FW^T),
% \end{equation}
% with $F\in \C^{(2N_x+1) \times (2N_t+1)}$.
For $u$ solution of (\ref{eq:pbd}), we denote by $U = (u_{kl})_{|k|\leq N_x, \; |l|\leq N_t} \in \C^{(2N_x+1)\times (2N_t+1)}$ the matrix such that
$$
u = \sum_{k=-N_x}^{N_x} \sum_{l=-N_t}^{N_t} u_{kl} e_k^x \otimes e_l^t,
$$
and by $F = (f_{kl})_{|k|\leq N_x, \; |l|\leq N_t} \in \C^{(2N_x+1)\times (2N_t+1)}$ the matrix such that
$$
f = \sum_{k=-N_x}^{N_x} \sum_{l=-N_t}^{N_t} f_{kl} e_k^x \otimes e_l^t.
$$
In this discrete setting, the first problem is equivalent to: find $U\in \C^{(2N_x+1)\times (2N_t+1)}$ such that 
$$
D_xU + b_x N_x U + U = F
$$
and the second problem equivalent to: find $U\in \C^{(2N_x+1)\times (2N_t+1)}$ such that 
\begin{equation}\label{eq:pbdisc2}
D_x U + b_x N_x U + U D_t +  b_t UN_t + U = F
\end{equation}
where $D_x, N_x \in \C^{(2N_x+1) \times (2N_x+1)}$, $D_t, N_t \in \C^{(2N_t +1) \times (2N_t+1)}$, and for all $-N_x \leq k, k' \leq N_x$ and $-N_t \leq l,l' \leq N_t$, 
\begin{eqnarray*}
(D_x)_{kk'} & = & |k|^2 \delta_{kk'}, \\
(D_t)_{ll'} & = & |l|^2 \delta_{ll'},\\
(N_x)_{kk'} & = & i k \delta_{kk'}, \\
(N_t)_{ll'} & = & il \delta_{ll'}.\\
\end{eqnarray*}

\subsection{Tests with the Decomposition algorithm}

Let us begin with the numerical tests performed to illustrate the convergence of the Decomposition algorithm presented in Section~\ref{sec:us}.

\subsubsection{The fixed-point loop}\label{sec:fpcrit}

Before presenting the numerical results obtained with this algorithm, we would like to discuss the fixed-point procedure used in practice in order to compute the pair of functions
$(r_n,s_n)\in V^{N_t}_t\times V^{N_x}_x$, solution of (\ref{eq:explicitalg}) at each iteration $n\in\N^*$ of the Decomposition algorithm. 

The algorithm reads as follows:
\begin{itemize}
 \item choose $\left( r_n^{(0)}, s_n^{(0)} \right) \in V^{N_t}_t \times V^{N_x}_x$ and set $m=1$;
\item find $\left( r_n^{(m)}, s_n^{(m)} \right) \in V^{N_t}_t \times V^{N_x}_x$ such that
$$
\left\{
\begin{array}{l}
 b_s\left(r_n^{(m)} \otimes s_n^{(m-1)}, \delta r \otimes s_n^{(m-1)}\right) = l\left(\delta r \otimes s_n^{(m-1)}\right) - a\left( u_{n-1}, \delta r \otimes s_n^{(m-1)}\right), \; \forall \delta r\in V^{N_x}_x,\\
b_s\left(r_n^{(m)} \otimes s_n^{(m)}, r_n^{(m)} \otimes \delta s\right) = l\left(r_n^{(m)} \otimes \delta s\right) - a\left( u_{n-1}, r_n^{(m)} \otimes \delta s \right), \; \forall \delta s\in V^{N_t}_t;\\
\end{array}
\right .
$$
\item set $m=m+1$.
\end{itemize}

This fixed-point procedure exhibits of the exponential convergence rate which is numerically observed for standard greedy algorithms for symmetric coercive problems, as discussed in 
Section~\ref{sec:sym}. The stopping criterion we choose for simulations presented in Section~\ref{sec:dectest} is the following: $\|r_n^{(m)}\otimes s_n^{(m)} - r_n^{(m-1)}\otimes s_n^{(m-1)}\| < \varepsilon$ where $\|\cdot\|$ 
denotes the Frobenius norm and $\varepsilon = 10^{-8}$.  

\subsubsection{Numerical results}\label{sec:dectest}

We consider the second problem presented in Section~\ref{sec:toy}, where $f\in L^2_{\rm per}((-1,1)^2, \C)$ is chosen such that
$$
f = \sum_{k\in \Z} \sum_{l \in \Z} f_{kl} e_k^x \otimes e_l^t,
$$
with $f_{kl} = \frac{1}{|k|^2 + |l|^2 +1}$ for all $k,l\in \Z$. 

We decompose the bilinear form $a(\cdot, \cdot)$ as 
$a(\cdot, \cdot) = b_s(\cdot, \cdot) + b(\cdot, \cdot)$ where $b_s(\cdot, \cdot) = a_s(\cdot, \cdot)$ is the symmetric part of $a(\cdot, \cdot)$ defined in (\ref{eq:sympartdef}) and 
$b(\cdot, \cdot) = a_{as}(\cdot, \cdot)$ is the antisymmetric part of $a(\cdot, \cdot)$ defined in (\ref{eq:antisympartdef}).

\medskip

Let us recall that Proposition~\ref{prop:explicit} states that, in the finite-dimensional case, there exists a rate $\kappa$ (which depends on the dimension of the Hilbert spaces) small 
enough such that if $\|\cB_s^{-1} \cB\|_{\mathfrak{L}(V,V)} < \kappa$, the Decomposition algorithm is ensured to converge. In the numerical simulations presented below, we have witnessed that there 
exists a threshold rate $\kappa$ such that 
\begin{itemize}
 \item if $\|\cB_s^{-1} \cB\|_{\mathfrak{L}(V,V)} < \kappa$, the algorithm converges; 
\item if $\|\cB_s^{-1} \cB\|_{\mathfrak{L}(V,V)} = \kappa$, the algorithm does not converge, but the norm of the residual remains bounded; 
\item if $\|\cB_s^{-1} \cB\|_{\mathfrak{L}(V,V)} > \kappa$, the algorithm does not converge, and the norm of the residual blows up. 
\end{itemize}
Besides, the rate $\kappa$ seems to be independent of the dimension of the Hilbert spaces. 

\medskip

For $n\in\N^*$, let $U_n\in \C^{(2N_x+1)\times (2N_t+1)}$ denote the approximation of $U$ solution of (\ref{eq:pbdisc2}) given by the algorithm at the $n^{th}$ iteration. 
The following three figures show the evolution of the logarithm of the norm of the residual $\|F - D_x U_n - U_n D_t - b_x N_x U_n - b_t U_n N_t - U \|_{\mathfrak{S}_2}$ 
(where $\|\cdot\|_{\mathfrak{S}_2}$ denotes 
the Frobenius norm), for different values of 
$N = N_x = N_t$ and $b= b_x = b_t$ as a function of $n$. We observe numerically that in this case, the limiting 
rate is obtained for $b = 1.5$ for any value of $N$. 

\begin{figure}
 \begin{center}
 \includegraphics[width = 12cm]{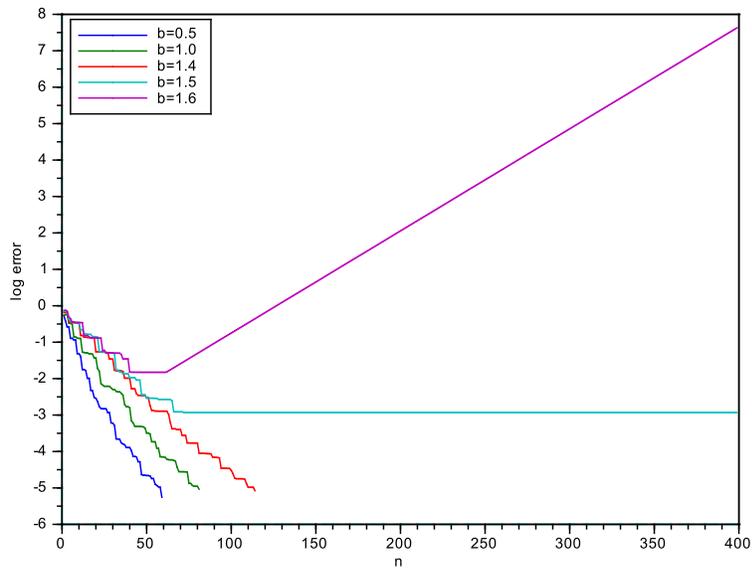}
 % N=10_dec.eps: 0x0 pixel, 300dpi, 0.00x0.00 cm, bb=14 220 581 621
\caption{Evolution of $\log_{10}\left(\|F - D_x U_n - U_n D_t - b_x N_x U_n - b_t U_n N_t - U \|_{\mathfrak{S}_2}\right)$ as a function of $n$ for $N=20$ and different values of $b$.}
\end{center}
\end{figure}

\begin{figure}
\begin{center}
 \includegraphics[width = 12cm]{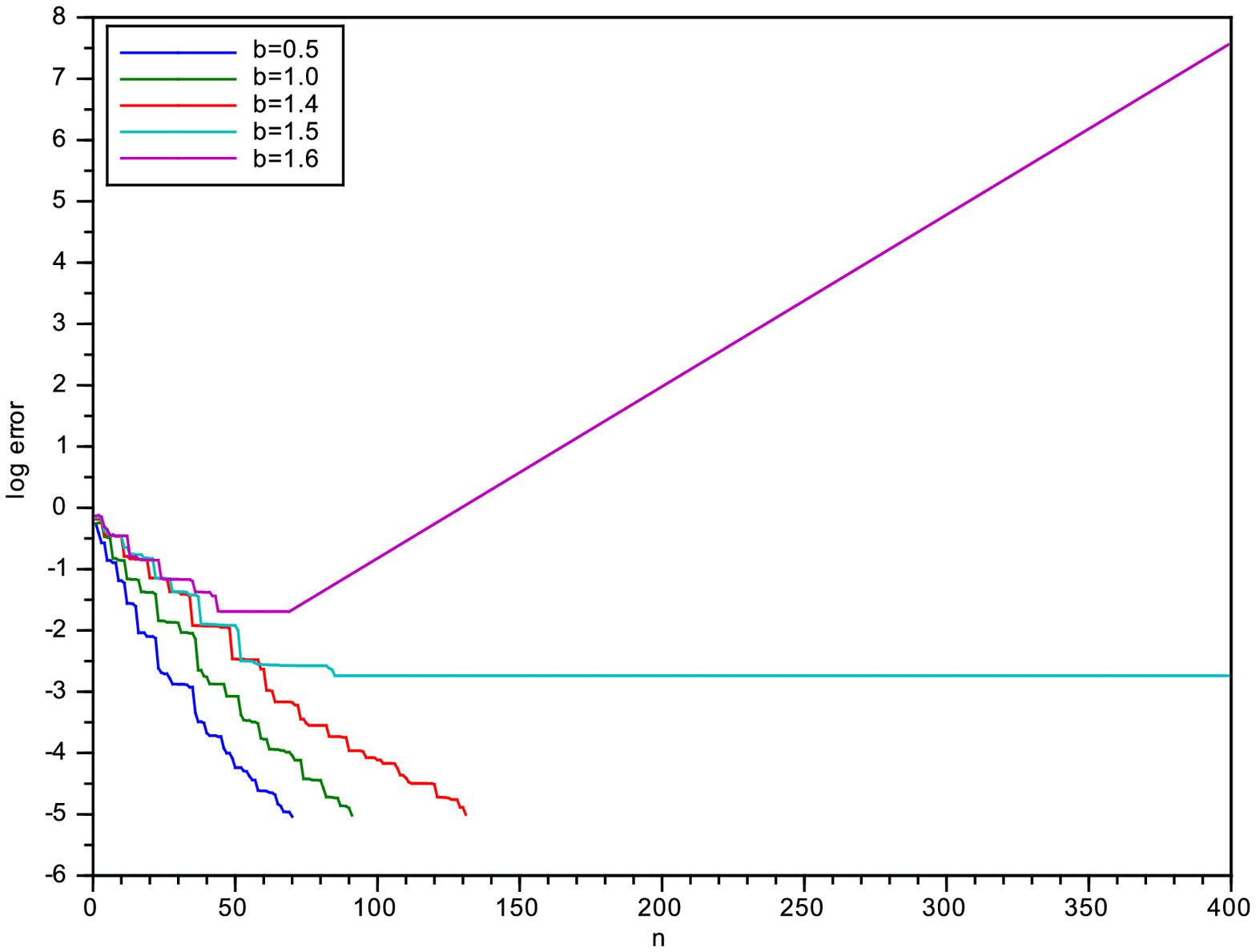}
 % N=10_dec.eps: 0x0 pixel, 300dpi, 0.00x0.00 cm, bb=14 220 581 621

\caption{Evolution of $\log_{10}\left(\|F - D_x U_n - U_n D_t - b_x N_x U_n - b_t U_n N_t - U \|_{\mathfrak{S}_2}\right)$ as a function of $n$ for $N=50$ and different values of $b$.}
\end{center}
\end{figure}

\begin{figure}
\begin{center}
 \includegraphics[width = 12cm]{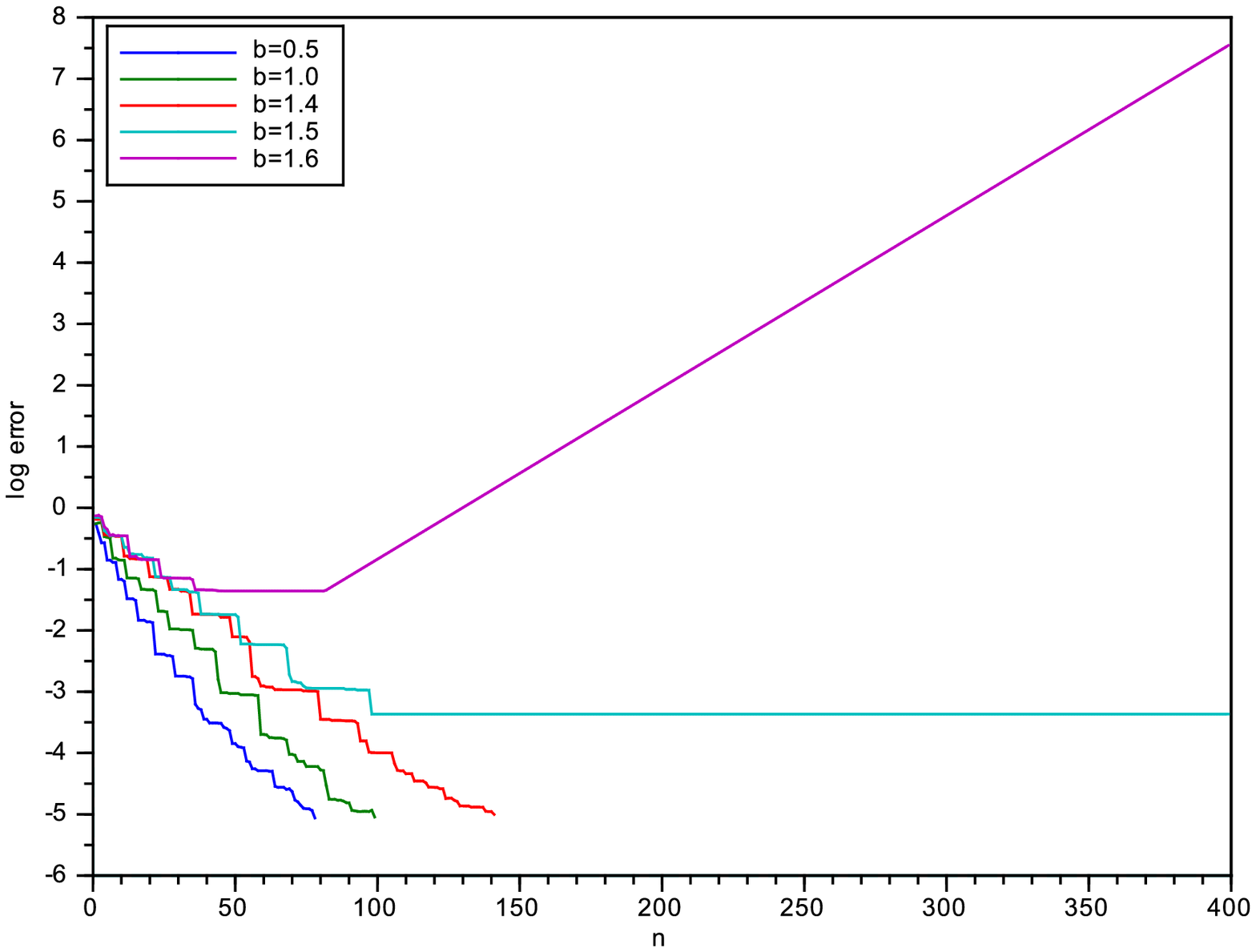}
 % N=10_dec.eps: 0x0 pixel, 300dpi, 0.00x0.00 cm, bb=14 220 581 621

\caption{Evolution of $\log_{10}\left(\|F - D_x U_n - U_n D_t - b_x N_x U_n - b_t U_n N_t - U \|_{\mathfrak{S}_2}\right)$ as a function of $n$ for $N=100$ and different values of $b$.}
\end{center}
\end{figure}

\medskip

We also performed another set of tests with a modified version of the Decomposition algorithm, in order to \itshape increase the threshold rate $\kappa$\normalfont, in a sense which will 
be precised below. The modified Decomposition algorithm reads as follows for $\alpha \in [1, +\infty)$:   
\begin{enumerate}
 \item let $u_0 = 0$ and $n=1$;
\item find $(r_n, s_n)\in V_x \times V_t$ such that
\begin{equation}\label{eq:explicitalg2}
(r_n,s_n) \in \mathop{\mbox{\rm argmin}}_{(r,s)\in V_x \times V_t} \frac{\alpha}{2} b_s( r\otimes s, r\otimes s) - l(r\otimes s) -a(u_{n-1}, r\otimes s);  
\end{equation}
\item set $u_n = u_{n-1} + r_n \otimes s_n$ and set $n=n+1$.
\end{enumerate}
Let us point out that this algorithms is equivalent to the standard Decomposition algorithm presented in Section~\ref{sec:us} in the case when $\alpha = 1$. 

\medskip

Equivalently, for each $n\in\N^*$, $r_n\otimes s_n$ is a tensor product solution to the first iteration of the greedy algorithm applied to 
the symmetric coercive problem
$$
\left\{
\begin{array}{l}
 \mbox{find }u\in V\mbox{ such that}\\
\forall v\in V, \; \alpha b_s(u,v) = l(v) - a(u_{n-1}, v).\\
\end{array}
\right .
$$

We observe that this algorithm has the same convergence properties as the standard Decomposition algorithm, i.e. there exists a threshold rate 
$\kappa_\alpha = \|\cB_s^{-1}\cB\|_{\cL(V)}$ such that
\begin{itemize}
 \item if $\|\cB_s^{-1} \cB\|_{\mathfrak{L}(V,V)} < \kappa_\alpha$, the algorithm converges; 
\item if $\|\cB_s^{-1} \cB\|_{\mathfrak{L}(V,V)} = \kappa_\alpha$, the algorithm does not converge, but the norm of the residual remains bounded; 
\item if $\|\cB_s^{-1} \cB\|_{\mathfrak{L}(V,V)} > \kappa_\alpha$, the algorithm does not converge, and the norm of the residual blows up. 
\end{itemize}
The rate $\kappa_\alpha$ also seems not to depend on the dimension of the Hilbert space. Besides, $\alpha \in [1,+\infty) \mapsto \kappa_\alpha$ seems to 
be an increasing function. Thus, choosing a larger value of $\alpha$ seems to lead to an algorithm which is convergent for larger values of $b$. 
However, the larger $\alpha$, the smaller the rate of convergence of the algorithm for a given value of $N$ and $b$. 

\medskip

The figure below presents the evolution of the logarithm of the norm of the residual 
$\|F - D_x U_n - U_n D_t - b_x N_x U_n - b_t U_n N_t - U \|_{\mathfrak{S}_2}$ for the second problem for $\alpha = 2$, $N=50$ and different values of $b$. The 
threshold value of $b$ seems to be in this case $b=2.6$.  

\begin{figure}
\begin{center}
 \includegraphics[width = 12cm]{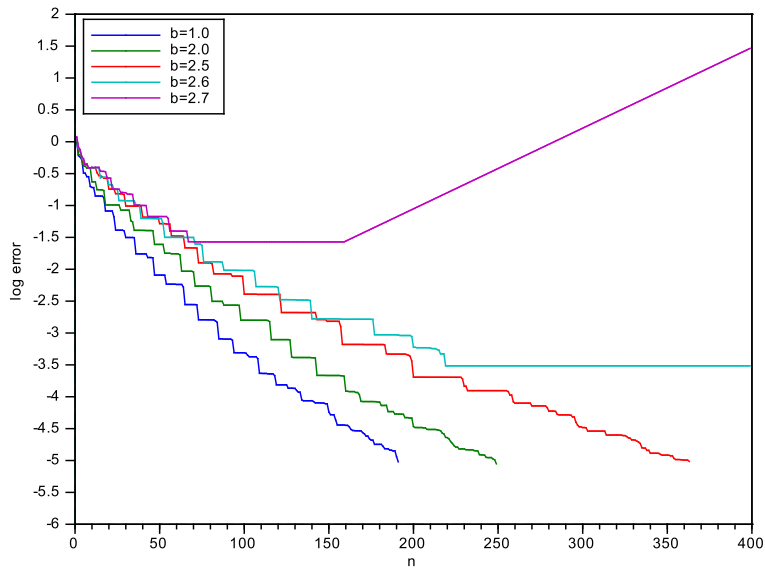}
 % N=10_dec.eps: 0x0 pixel, 300dpi, 0.00x0.00 cm, bb=14 220 581 621

\caption{Evolution of $\log_{10}\left(\|F - D_x U_n - U_n D_t - b_x N_x U_n - b_t U_n N_t - U \|_{\mathfrak{S}_2}\right)$ as a function of $n$ for the modified Decomposition algorithm with 
$\alpha = 2$, $N=50$ and different values of $b$.}
\end{center}
\end{figure}

\subsection{Comparison of Galerkin, Minimax, Dual Greedy and Decomposition algorithms}

In this section, we present various numerical tests performed to compare the performances of the Galerkin, Minimax, Dual Greedy and Decomposition algorithms. 

\subsubsection{Fixed point procedures}\label{sec:fpalgos}

Before presenting the simulations done to compare the performance of the four algorithms, we will detail more precisely the fixed-point procedures used for the Galerkin, Dual and MiniMax 
algorithms.

\medskip

Let us first present the fixed-point procedure used in the Galerkin algorithm. For all $n\in\N^*$, the method used to compute a pair $(r_n,s_n)\in V_x\times V_t$ solution to
 (\ref{eq:algnaive}) is the following: 
\begin{itemize}
 \item choose $\left( r_n^{(0)}, s_n^{(0)}\right) \in V_x\times V_t$ and set $m=1$;
\item find $\left( r_n^{(m)}, s_n^{(m)}\right) \in V_x\times V_t$ such that
$$
\left\{
\begin{array}{l}
 a\left(r_n^{(m)} \otimes s_n^{(m-1)}, \delta r \otimes s_n^{(m-1)}\right) = l\left(\delta r \otimes s_n^{(m-1)}\right) - a\left( u_{n-1}, \delta r \otimes s_n^{(m-1)}\right), \; \forall \delta r\in V^{N_x}_x,\\
a\left(r_n^{(m)} \otimes s_n^{(m)}, r_n^{(m)} \otimes \delta s\right) = l\left(r_n^{(m)} \otimes \delta s\right) - a\left( u_{n-1}, r_n^{(m)} \otimes \delta s \right), \; \forall \delta s\in V^{N_t}_t;\\
\end{array}
\right .
$$
\item set $m=m+1$.
\end{itemize}
All the iterations of the fixed-point procedure are well-defined. However, we have seen from Example~\ref{ex:continuous} that there are cases where any solution $(r_n\otimes s_n) \in V_x\times V_t$ of
$$
\forall (\delta r, \delta s)\in V_x\times V_t, \; a(r_n\otimes s_n, r_n\otimes \delta s + \delta r \otimes s_n) = l(r_n\otimes \delta s + \delta r \otimes s_n) -a(u_{n-1}, r_n\otimes \delta s + \delta r \otimes s_n),
$$
satisfies $r_n\otimes s_n = 0$. From Example~\ref{ex:continuous}, this is the case in particular for the first problem presented in Section~\ref{sec:toy} with $f\in L^2_{\rm per}((-1,1)^2, \C)$ 
being chosen such that for almost all $x,t\in (-1,1)$, $f(x,t) = \phi(x-t)$ with $\phi\in L^2_{\rm per}((-1,1), \C)$ a real-valued odd function. The function $f = \sum_{k,l\in\Z} f_{kl} e_k^x\otimes e_l^t$ with 
$f_{kl} = \delta_{1,k} \delta_{-1,l} + \delta_{-1,k} \delta_{1,l}$ for all $k,l\in\Z$ is an example of such a function. 
On this particular example, we observe numerically that, for $n=1$, the fixed-point procedure presented above does not converge. 

The practical implementation of the Galerkin algorithm requires the use of another stopping criterion as the one used in Section~\ref{sec:fpcrit}. In the numerical 
simulations presented in Section~\ref{sec:numrescomp}, and in the fixed-point procedures 
used in all the other algorithms implemented, the stopping criterion will be $m < m_{max}$ with $m_{max} = 20$.   
   
\medskip

The fixed-point algorithm used for the Decomposition algorithm has been detailed in Section~\ref{sec:fpcrit}, and the one used for the MiniMax 
algorithm has been described in Section~\ref{sec:minimax}. The fixed-point procedure for the Dual Greedy algorithm reads as follows:
\begin{itemize}
 \item choose $\left(\rt_n^{(0)}, \st_n^{(0)}\right) \in V_x \times V_t$ and set $\mt = 1$;
\item while $\mt < m_{max}$, find $\left(\rt_n^{(m)}, \st_n^{(m)}\right) \in V_x\times V_t$ such that for all $(\delta \rt, \delta \st)\in V_x\times V_t$, 
$$
\left\{
\begin{array}{l}
 \langle \rt_n^{(m)} \otimes \st_n^{(m-1)},  \delta \rt \otimes \st_n^{(m-1)}\rangle_V = l(\delta \rt \otimes \st_n^{(m-1)}) - a(u_{n-1} , \delta \rt \otimes \st_n^{(m-1)}),\\
\langle \rt_n^{(m)} \otimes \st_n^{(m)},  \rt_n^{(m)} \otimes \delta \st\rangle_V = l(\rt_n^{(m)} \otimes \delta \st) - a(u_{n-1} , \rt_n^{(m)} \otimes \delta \st);\\
\end{array}
\right.
$$
\item set $\mt = \mt+1$; 
\item if $\mt \geq m_{max}$, set $\rt_n = \rt_n^{(\mt)}$ and $\st_n = \st_n^{(\mt)}$; 
\item choose $\left(r_n^{(0)}, s_n^{(0)}\right) \in V_x \times V_t$ and set $m = 1$;
\item while $m < m_{max}$, find $\left(r_n^{(m)}, s_n^{(m)}\right) \in V_x\times V_t$ such that for all $(\delta \rt, \delta \st)\in V_x\times V_t$, 
$$
\left\{
\begin{array}{l}
a(r_n^{(m)}\otimes s_n^{(m-1)}, \delta \rt \otimes \st_n) = l(\delta \rt \otimes \st_n) - a(u_{n-1} , \delta \rt \otimes \st_n),\\
a(r_n^{(m)}\otimes s_n^{(m)}, \rt_n \otimes \delta \st) = l(\rt_n \otimes \delta \st) - a(u_{n-1} , \rt_n \otimes \delta \st);\\
\end{array}
\right.
$$
\item set $m = m +1$; 
\item if $m\geq m_{max}$, set $r_n = r_n^{(m)}$ and $s_n = s_n^{(m)}$. 
\end{itemize}

\subsubsection{Numerical results}\label{sec:numrescomp}

We present here some numerical results obtained for the second problem introduced in Section~\ref{sec:toy}, with $f$ chosen as in Section~\ref{sec:dectest}. Here $N = N_t = N_x = 50$ 
and the different figures show the rates of convergence of the different algorithms for several values of $b = b_x = b_t$. The different values of $b$ are the following: $0.01$, 
$0.1$, $1$ and $2$. When $b=2$, the Decomposition algorithm does not converge. 

\begin{figure}
\begin{center}
 \includegraphics[width = 12cm]{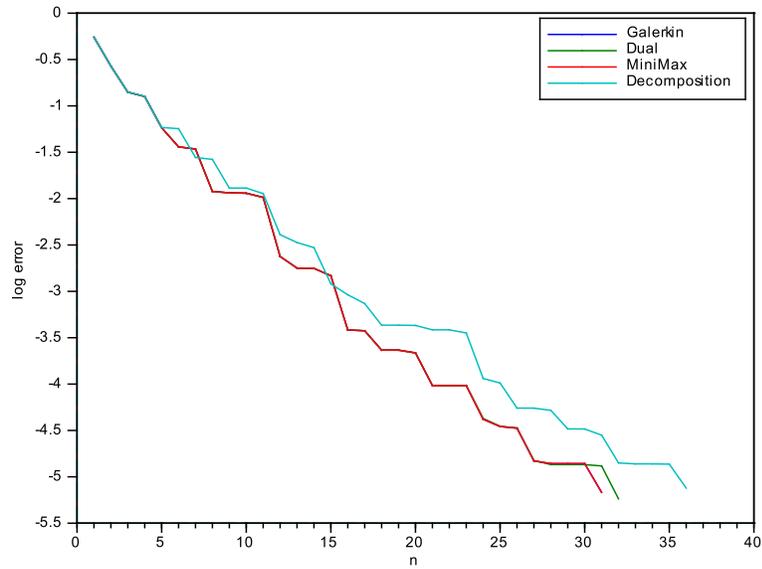}
 % N=10_dec.eps: 0x0 pixel, 300dpi, 0.00x0.00 cm, bb=14 220 581 621

\caption{Evolution of $\log_{10}\left(\|F - D_x U_n - b_x N_x U_n - U \|_{\mathfrak{S}_2}\right)$ for the different algorithms as a function of $n$ 
with $N=50$ and $b= 0.01$.}
\end{center}
\end{figure}

\begin{figure}
\begin{center}
 \includegraphics[width = 12cm]{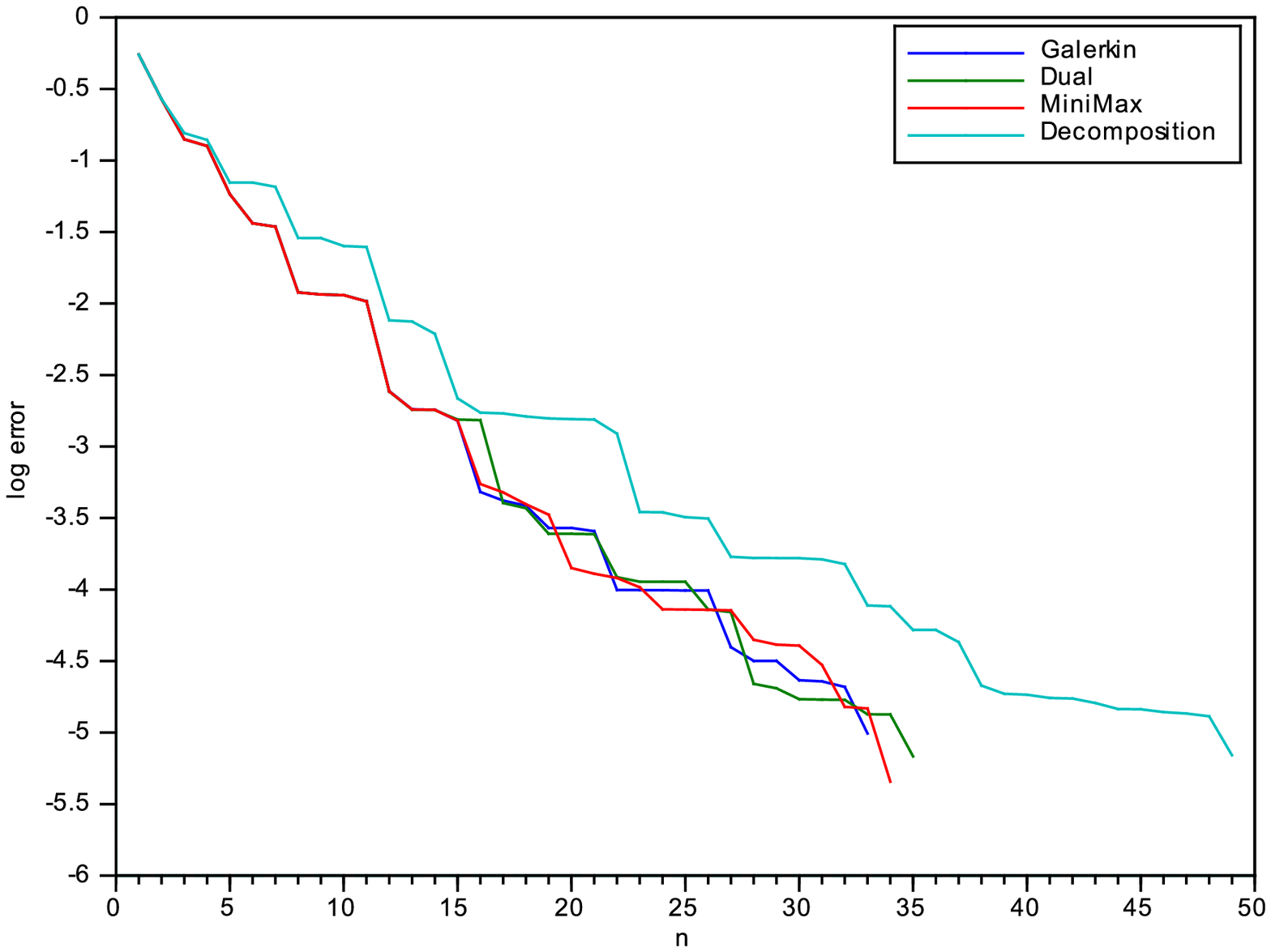}
 % N=10_dec.eps: 0x0 pixel, 300dpi, 0.00x0.00 cm, bb=14 220 581 621

\caption{Evolution of $\log_{10}\left(\|F - D_x U_n - b_x N_x U_n - U \|_{\mathfrak{S}_2}\right)$ for the different algorithms as a function of $n$ 
with $N=50$ and $b= 0.1$.}
\end{center}
\end{figure}

\begin{figure}
\begin{center}
 \includegraphics[width = 12cm]{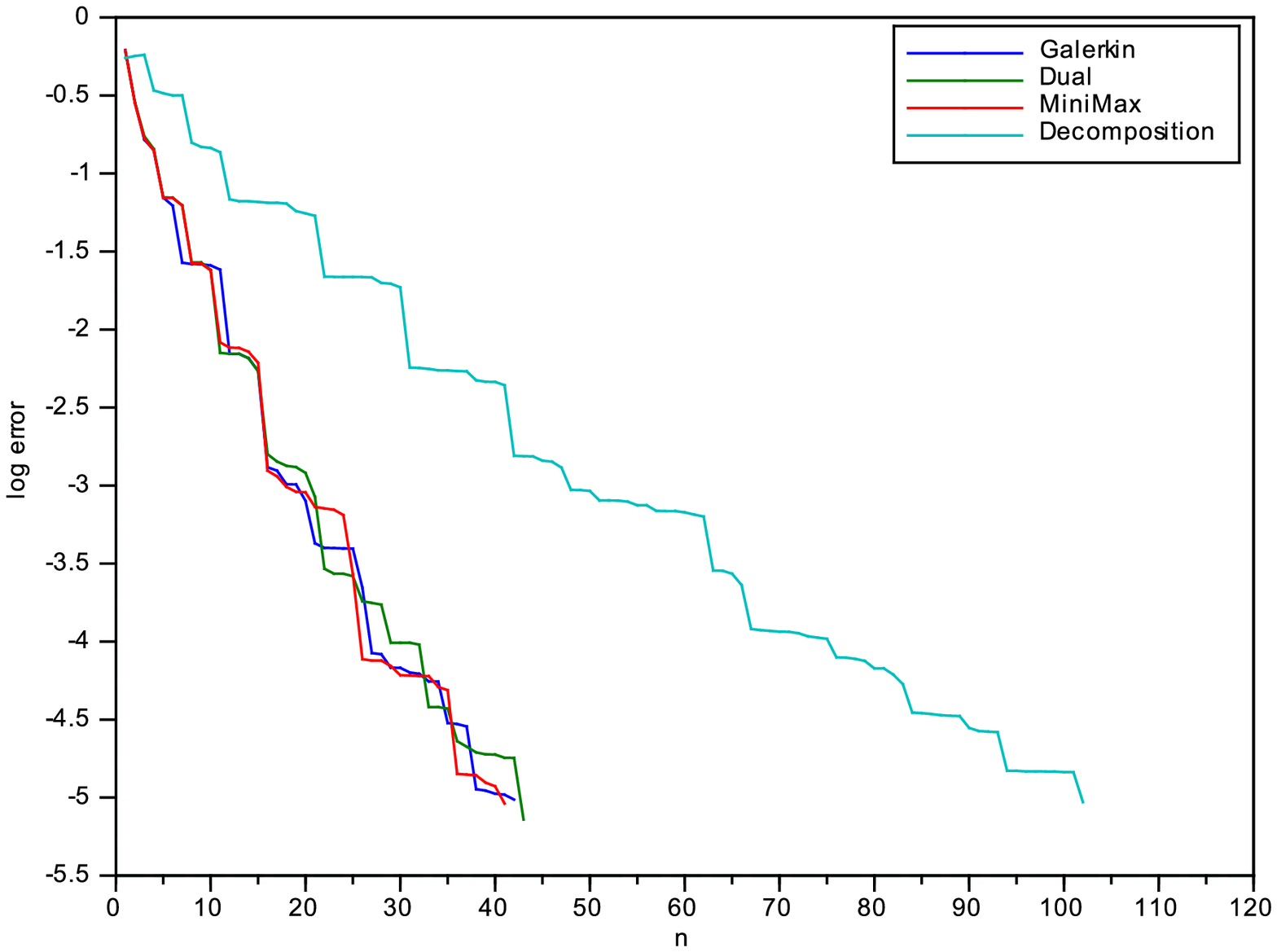}
 % N=10_dec.eps: 0x0 pixel, 300dpi, 0.00x0.00 cm, bb=14 220 581 621

\caption{Evolution of $\log_{10}\left(\|F - D_x U_n - b_x N_x U_n - U \|_{\mathfrak{S}_2}\right)$ for the different algorithms as a function of $n$ 
with $N=50$ and $b= 1$.}
\end{center}
\end{figure}

\begin{figure}
\begin{center}
 \includegraphics[width = 12cm]{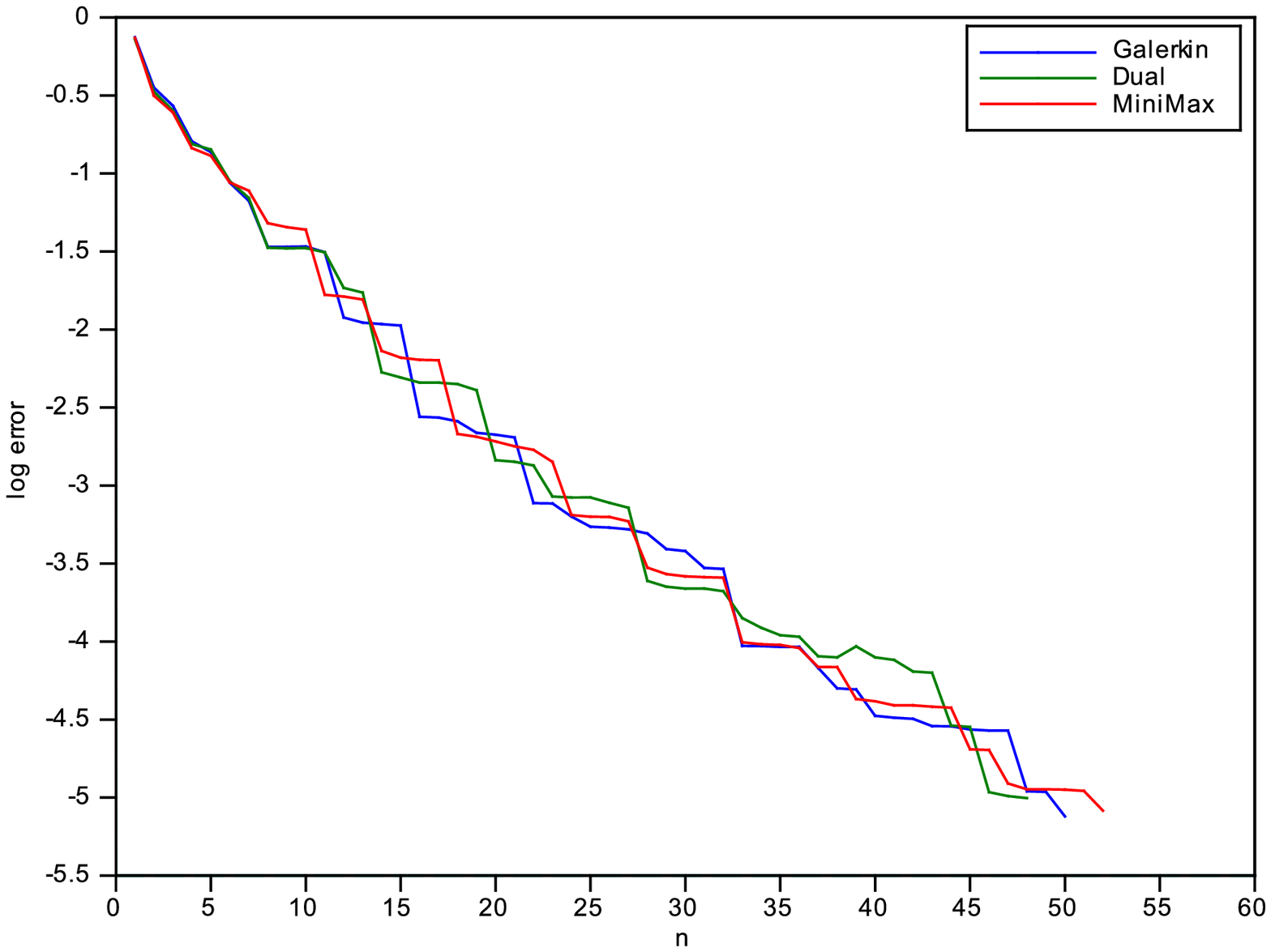}
 % N=10_dec.eps: 0x0 pixel, 300dpi, 0.00x0.00 cm, bb=14 220 581 621

\caption{Evolution of $\log_{10}\left(\|F - D_x U_n - b_x N_x U_n - U \|_{\mathfrak{S}_2}\right)$ for the different algorithms as a function of $n$ 
with $N=50$ and $b = 2$.}
\end{center}
\end{figure}

We observe numerically that even if the Galerkin algorithm is not well-defined through the use of the equations (\ref{eq:algnaive}), using a fixed-point procedure as described above 
and stopping the procedure after a number $m_{max}$ of iterations is enough to observe good convergence properties of the algorithm. 
Besides, the Dual Greedy and MiniMax algorithm are almost as efficient as the Galerkin algorithm. 

\medskip

The Decomposition algorithm is very efficient when the antisymmetric part of the bilinear form $a(\cdot, \cdot)$ is small, but performs badly when $b$ becomes larger. 
Let us point out 
that the CPU time needed to compute one tensor product in the MiniMax or the Dual Greedy algorithms is twice the time needed in the Decomposition or the Galerkin algorithms. 
Thus, the Decomposition algorithm is more efficient than the Dual or the MiniMax algorithm for small antisymmetric parts of the bilinear form $a(\cdot, \cdot)$ but, when 
$b$ becomes too large, the algorithm behaves poorly.

\section{Appendix: other algorithms}\label{sec:app}

In this section, we present two other possible tracks towards the design of efficient greedy algorithms for high-dimensional non-symmetric problems, for 
which there is still work in progress.  

\subsection{An ill-defined (but converging) algorithm}

In this section, we first present an algorithm whose iterations are ill-defined in general but for which we can prove the convergence in the full general case. 
Of course, this algorithm will not be useful in practice but we believe the proof is instructive in our context.

Let $\alpha>0$. The algorithm reads as follows: 
\begin{enumerate}
 \item set $u_0 = 0$ and $n=1$; 
\item find $(r_n, \rt_n, s_n, \st_n)\in V_x^2 \times V_t^2$ such that
\begin{equation}\label{eq:illdefalgo}
\left\{
\begin{array}{l}
\dps (\rt_n, \st_n) \in \mathop{\rm argmin}_{(\rt, \st)\in V_x \times V_t} \frac{1}{2}\|\rt\otimes \st\|_V^2 - l(\rt\otimes \st) - a(u_{n-1} + r_{n}\otimes s_{n}, \rt\otimes \st),\\
\dps (r_n,s_n) \in \mathop{\rm argmin}_{(r,s)\in V_x \times V_t} \frac{\alpha}{2}\|r\otimes s\|_V^2 - a(r\otimes s, \rt_n \otimes \st_n);\\
\end{array}
\right .
\end{equation}
\item set $u_n = u_{n-1} + r_n\otimes s_n$ and $n=n+1$. 
\end{enumerate}

The Euler equations associated to these coupled minimization problems read: for all $(\delta r, \delta \rt, \delta s, \delta \st) \in V_x^2 \times V_t^2$, 
\begin{equation}\label{eq:ELas}
\left\{
\begin{array}{l}
 \langle \rt_n \otimes \st_n, \rt_n \otimes \delta \st + \delta \rt \otimes \st_n \rangle_V = l( \rt_n \otimes \delta \st + \delta \rt \otimes \st_n) - a(u_{n-1} + r_n\otimes s_n,  \rt_n \otimes \delta \st + \delta \rt \otimes \st_n),\\
\alpha \langle r_n\otimes s_n, r_n \otimes \delta s + \delta r \otimes s_n\rangle_V = a(r_n \otimes \delta s + \delta r \otimes s_n, \rt_n \otimes \st_n).\\
\end{array}
\right .
\end{equation}

Using Lemma~\ref{lem:maxminform} and (\ref{eq:touse}), these definitions imply that for all $n\in\N^*$, 
\begin{equation}\label{eq:optrt}
\|\rt_n\otimes\st_n\|_V = \mathop{\sup}_{(\rt, \st)\in V_x \times V_t} \frac{ l(\rt \otimes \st) - a(u_{n-1} + r_n\otimes s_n, \rt\otimes \st)}{\|\rt \otimes \st\|_V}
\end{equation}
and
\begin{equation}\label{eq:optr}
\alpha \|r_n\otimes s_n\|_V = \mathop{\sup}_{(r,s)\in V_x\times V_t} \frac{a(r\otimes s, \rt_n\otimes \st_n)}{\|r\otimes s\|_V}.
\end{equation}
Indeed, for the first equality, one has to consider the symmetric coercive continuous bilinear form $a^1(\cdot, \cdot) =\langle \cdot, \cdot \rangle_V$ and the continuous 
linear form $l^1(\cdot)= l(\cdot) - a(u_{n-1} + r_n\otimes s_n,\cdot)$. For the second equality, the symmetric coercive continuous bilinear form to consider is 
$a^2(\cdot, \cdot) = \alpha \langle \cdot, \cdot \rangle_V$ and the continuous linear form is $l^2(\cdot) = a(\cdot, \rt_n\otimes \st_n)$.

The following result holds:
\begin{prop}\label{prop:illdef}
 Let us assume that the iterations of algorithm (\ref{eq:illdefalgo}) are well-defined for all $n\in\N^*$. Then, $(u_n)_{n\in\N^*}$ converges 
to $u$ in the sense of the injective norm: 
$$
\|u -u_n\|_{i\cA} = \mathop{\sup}_{(\rt,\st)\in V_x \times V_t} \frac{a(u-u_n, \rt\otimes \st)}{\|\rt\otimes \st\|_V} \mathop{\longrightarrow}_{n \to \infty} 0. 
$$
\end{prop}

\begin{proof}
 Let us first remark that, using (\ref{eq:optrt}), 
$$
\|u-u_n\|_{i\cA} = \mathop{\sup}_{(\rt,\st)\in V_x \times V_t} \frac{a(u-u_n, \rt\otimes \st)}{\|\rt\otimes \st\|_V} = \|\rt_n\otimes \st_n\|_V.
$$
Thus, it is sufficient to prove that the sequence $(\|\rt_n \otimes \st_n \|_V)_{n\in\N^*}$ converges to $0$ as $n$ goes to infinity. Let us first prove that this sequence is non-increasing. 
Let $n\geq 2$. From the Euler equation (\ref{eq:ELas}) associated to the minimization problem defining $(\rt_n, \st_n)$, it holds that
\begin{equation}\label{eq:ELdeclined}
\|\rt_n \otimes \st_n\|_V^2 = l(\rt_n \otimes \st_n) - a(u_{n-1} + r_n\otimes s_n, \rt_n\otimes \st_n).
\end{equation}
Besides, using (\ref{eq:optrt}) at iteration $n-1$, we obtain
\begin{eqnarray*}
\|\rt_{n-1} \otimes \st_{n-1}\|_V & = & \mathop{\sup}_{(\rt, \st)\in V_x \times V_t, \; \rt\otimes \st\neq 0} 
\frac{l(\rt\otimes \st) -a(u_{n-1}, \rt\otimes \st)}{\|\rt\otimes \st\|_V}\\
& \geq & \frac{l(\rt_n\otimes \st_n) - a(u_{n-1}, \rt_n\otimes \st_n)}{\|\rt_n\otimes \st_n\|_V},\\
\end{eqnarray*}
which, using (\ref{eq:ELdeclined}), leads to
$$
\|\rt_n \otimes \st_n\|_V^2 \leq \|\rt_{n-1} \otimes \st_{n-1}\|_V \|\rt_n \otimes \st_n\|_V -a( r_n\otimes s_n, \rt_n\otimes \st_n).
 $$
Using the second Euler equation (\ref{eq:ELas}) defining $(r_n,s_n)$, we have that $\alpha \|r_n\otimes s_n\|_V^2  = a(r_n\otimes s_n, \rt_n\otimes \st_n)$. Finally, 
it holds that 
\begin{equation}\label{eq:neweq}
\|\rt_n \otimes \st_n\|_V^2 \leq \|\rt_n \otimes \st_n\|_V \|\rt_{n-1} \otimes \st_{n-1}\|_V - \alpha \|r_n \otimes s_n\|_V^2. 
\end{equation}
This implies that the sequence $(\|\rt_n \otimes \st_n\|_V)_{n\in\N^*}$ is non-increasing and thus converges towards a limit $z\geq 0$. Let us argue by contradiction and assume 
that $z>0$. Dividing by $\|\rt_n\otimes \st_n\|_V$ equation (\ref{eq:neweq}), we obtain
$$
\alpha\frac{ \|r_n\otimes s_n\|_V^2}{\|\rt_n\otimes \st_n\|_V} \leq \|\rt_{n-1} \otimes \st_{n-1} \|_V - \|\rt_n \otimes \st_n\|_V. 
$$
Since we have assumed that $\|\rt_n\otimes \st_n\|_V \geq z>0$ for all $n\in\N^*$, the series of general term $(\|r_n\otimes s_n\|_V^2)_{n\in\N^*}$ converges and 
$\dps \|r_n\otimes s_n\|_V \mathop{\longrightarrow}_{n\to \infty} 0$. Using (\ref{eq:optr}), this implies that for all $(r,s)\in V_x \times V_t$, 
$$
a(r\otimes s, \rt_n\otimes \st_n) \mathop{\longrightarrow}_{n\to \infty} 0.
$$
Using assumption (A1) and the fact that $(\|\rt_n \otimes \st_n\|_V)_{n\in\N^*}$ is bounded, we have
$$
\forall w\in V, \; a(w, \rt_n \otimes \st_n) \mathop{\longrightarrow}_{n\to \infty} 0.
$$
Since we have assumed that the operator $\cA$ is bijective, it is surjective on $V$, and the sequence $(\rt_n\otimes \st_n)_{n\in\N^*}$ weakly converges to $0$ in $V$. 

Using (\ref{eq:optrt}), it holds that for all $n\in\N^*$, 
$$
\|\rt_n \otimes \st_n\|_V^2 = l(\rt_n \otimes \st_n) - a(u_n, \rt_n \otimes \st_n) = l(\rt_n \otimes \st_n) - a\left( \sum_{k=1}^n r_k \otimes s_k, \rt_n \otimes \st_n\right).
$$
Since $(\rt_n\otimes\st_n)_{n\in\N^*}$ weakly converges to $0$ in $V$, necessarily $\dps l(\rt_n \otimes \st_n) \mathop{\longrightarrow}_{n\to\infty} 0$. 
Besides, using (\ref{eq:optr}), we have
\begin{eqnarray*}
\left| a\left( \sum_{k=1}^n r_k\otimes s_k, \rt_n\otimes \st_n \right)\right| & \leq & \sum_{k=1}^n \left| a\left( r_k\otimes s_k, \rt_n \otimes \st_n \right)\right|\\
& \leq & \alpha \sum_{k=1}^n \|r_k\otimes s_k\|_V \|r_n \otimes s_n\|_V\\
& \leq & \alpha \left( \sum_{k=1}^n \|r_k\otimes s_k\|_V^2 \right)^{1/2} (n \|r_n\otimes s_n\|_V^2)^{1/2}.\\
\end{eqnarray*}
Since the series of general term $(\|r_n\otimes s_n\|_V^2)_{n\in\N^*}$ is convergent, the sequence $\left(  \sum_{k=1}^n \|r_k\otimes s_k\|_V^2 \right)_{n\in\N^*}$ is bounded and 
there exists a subsequence of $(n\|r_n\otimes s_n\|_V^2)_{n\in\N^*}$ which converges to $0$. Thus, there exists a subsequence of $(\|\rt_n \otimes \st_n\|_V)_{n\in\N^*}$ converging to 
$0$ and since the whole sequence converges to $z$, we have $z=0$ by uniqueness of the limit. We obtain a contradiction.
\end{proof}

\begin{remark}
Unfortunately, as announced in the beginning of this section, in general, the iterations of algorithm (\ref{eq:illdefalgo}) are not well-defined in the sense that there may not exist 
a solution $(r_n, \rt_n, s_n, \st_n) \in V_x^2 \times V_t^2$ of the coupled minimization problems. 
Numerically, we can observe that if we use a coupled fixed-point algorithm similar to the one presented for the Minimax algorithm, 
the procedure does not converge in general. Finding a suitable way to adapt 
these ideas in an implementable well-defined algorithm is work in progress. 
\end{remark}

%  \begin{remark}
%   This algorithm is kind of an \normalfont implicit \itshape algorithm. A natural question is the following: could an \normalfont explicit \itshape version 
%  of this algorithm be convergent? In other words, we could define $(r_n, \rt_n, s_n, \st_n)\in V_x^2 \times V_t^2 $ at each iteration $n\in\N^*$ as
%  $$
%  \left\{
%  \begin{array}{l}
%  (\rt_n, \st_n) \in \mathop{\rm argmin}_{(\rt, \st)\in V_x \times V_t} \frac{1}{2}\|\rt\otimes \st\|_V^2 - l(\rt\otimes \st) - a(u_{n-1}, \rt\otimes \st)\\
%  (r_n,s_n) \in \mathop{\rm argmin}_{(r,s)\in V_x \times V_t} \frac{\alpha}{2}\|r\otimes s\|_V^2 - a(r\otimes s, \rt_n \otimes \st_n).\\
%  \end{array}
%  \right .
%  $$  
%  In this case, all the iterations of the algorithm are well-defined. Unfortunately, we observe numerically that such an algorithm does not converge. 
%  This is in accordance with the fact that we cannot prove anymore the decrease of the norm of the residual, unlike in the \normalfont implicit \itshape case.  
%  \end{remark}
\normalfont

\begin{remark}
When the bilinear form $a(\cdot, \cdot)$ is symmetric and coercive, all the iterations of algorithm (\ref{eq:illdefalgo}) are well-defined. If $\langle \cdot , \cdot \rangle_V$ is chosen 
to be equal to $a(\cdot, \cdot)$, the second equation of (\ref{eq:ELas}) implies that $r_n\otimes s_n = \frac{1}{\alpha}\rt_n\otimes \st_n$ and the first equation of (\ref{eq:ELas}) 
can be rewritten as: for all $(\delta \rt, \delta \st)\in V_x \times V_t$,
$$
\left( 1 + \frac{1}{\alpha} \right) \langle \rt_n \otimes \st_n, \rt_n \otimes \delta\st + \delta \rt \otimes \st_n \rangle_V  = l(\rt_n \otimes \delta \st + \delta \rt \otimes \st_n) - \langle u_{n-1}, \rt_n\otimes \delta \st + \delta \rt \otimes \st_n\rangle_V.
$$
This Euler equation is similar to the Euler equation of the first iteration of the standard greedy algorithm applied to the symmetric coercive problem:
$$
\left\{
\begin{array}{l}
 \mbox{find }\widetilde{u}\in V \mbox{ such that}\\
\forall v\in V,\left( 1 + \; \frac{1}{\alpha}\right) \langle \widetilde{u}, v\rangle_V = l(v) - \langle u_{n-1},v\rangle_V.\\  
\end{array}
\right .
$$
Thus, if we consider now the following (well-defined) algorithm
\begin{enumerate}
 \item set $u_0 = 0$ and $n=1$; 
\item find $(r_n,s_n)\in V_x \times V_t$ such that
$$
(r_n, s_n)\in \mathop{\mbox{argmin}}_{(r,s)\in V_x \times V_t} \frac{\lambda}{2}\|r\otimes s\|_V^2 -l(r\otimes s) - \langle u_{n-1}, v \rangle_V;
$$
\item set $u_n = u_{n-1} + r_n\otimes s_n$ and $n=n+1$,
\end{enumerate}
following the proof of Proposition~\ref{prop:illdef}, we can prove that the sequence $(u_n)_{n\in\N^*}$ converges to $\cL$ in the sense of the injective norm as soon as $\lambda >1$. 
Let us point out that when $\lambda = 1$, this algorithm is identical to the standard greedy algorithm applied to the symmetric coercive problem
$$
\left\{
\begin{array}{l}
 \mbox{find }\widetilde{u}\in V \mbox{ such that}\\
\forall v\in V, \langle \widetilde{u}, v\rangle_V = l(v).\\  
\end{array}
\right .
$$
\end{remark}

 \subsection{Link with a symmetric formulation}
 
 Let us now present another approach, for which no convergence result have been proved so far.
 The idea is based on the article~\cite{Cohennonsym} by Cohen, Dahmen and Welper, where the objective was to develop stable formulations of multiscale 
 convection-diffusion equations. 
 
 \medskip
 
 The principle of the method is to reformulate the antisymmetric problem (\ref{eq:nonsym}) defined on the Hilbert space $V$, as a symmetric problem defined on the Hilbert space 
 $V\times V$. Indeed, it is proved in \cite{Cohennonsym} that the unique solution of the problem
 \begin{equation}\label{eq:Cohen}
  \left\{
 \begin{array}{l}
  \mbox{find }(v,\vt)\in V\times V \mbox{ such that}\\
 a(w,\vt) = 0, \quad \forall w\in V,\\
 a(v,\wt) - \langle R_V \vt,\wt\rangle = l(\wt), \quad \forall \wt\in V, \\
\end{array}
 \right .
 \end{equation}
 is $(v,\vt) = (u,0)$ where $u$ is the unique solution of (\ref{eq:nonsym}). 
 
 \medskip

 This new problem is now \itshape symmetric\normalfont. It is equivalent to the following problem
 $$
 \left\{
 \begin{array}{l}
 \mbox{find }(v,\vt)\in V\times V\mbox{ such that}\\
 \left( \begin{array}{cc}
 0 & A^*\\
 A & -R_V \\
 \end{array}
 \right)
 \left( 
 \begin{array}{c}
  v\\
 \vt\\
 \end{array}
 \right)
  = 
 \left(
 \begin{array}{c}
  0\\
 L\\
 \end{array}
 \right) 
 \mbox{ in } V' \times V'.\\
 \end{array}
 \right .
 $$

 This new formulation of the problem is symmetric, but \itshape not coercive\normalfont. No convergence results exist for greedy methods in this framework. 
 However, the situation seems more encouraging than in the original non-symmetric case. It is to be noted though that the use of a simple Galerkin algorithm, similar to the one 
 introduced in Section~\ref{sec:galerkin}, does not work in this case 
 either. More subtle algorithms need to be designed in this case as well, and this is currently work in progress.

\bibliography{biblio}

\end{document}

%% file: tabsum2.pstex_t
\begin{picture}(0,0)%
\includegraphics{tabsum2.pstex}%
\end{picture}%
\setlength{\unitlength}{2693sp}%
\begingroup\makeatletter\ifx\SetFigFontNFSS\undefined%
\gdef\SetFigFontNFSS#1#2#3#4#5{%
  \reset@font\fontsize{#1}{#2pt}%
  \fontfamily{#3}\fontseries{#4}\fontshape{#5}%
  \selectfont}%
\fi\endgroup%
\begin{picture}(11679,8304)(529,-7723)
\put(721,-736){\makebox(0,0)[lb]{\smash{{\SetFigFontNFSS{10}{12.0}{\rmdefault}{\mddefault}{\updefault}{$L^2$}%
}}}}
\put(6121,-4741){\makebox(0,0)[lb]{\smash{{\SetFigFontNFSS{8}{9.6}{\rmdefault}{\mddefault}{\updefault}{$V = V_x\otimes V_t$}%
}}}}
\end{picture}%